\documentclass[10pt,reqno,twoside]{amsart}
\usepackage{amssymb,amsmath,amsthm,soul,color,paralist}
\usepackage{t1enc}
\usepackage{comment}
\usepackage[cp1250]{inputenc}
\usepackage{a4,indentfirst,latexsym}
\usepackage{graphics}
\usepackage[mathscr]{eucal}
\usepackage{cite,enumitem,graphicx}
\usepackage[colorlinks=true,urlcolor=blue,
citecolor=red,linkcolor=blue,linktocpage,pdfpagelabels,
bookmarksnumbered,bookmarksopen]{hyperref}
\usepackage[english]{babel}
\usepackage[left=2.50cm,right=2.50cm,top=2.72cm,bottom=2.72cm]{geometry}
\usepackage[colorinlistoftodos]{todonotes}
\usepackage[normalem]{ulem}

\allowdisplaybreaks

\numberwithin{equation}{section}

\newcommand{\R}{\mathbb{R}}

\newcommand{\h}{H^{s}(\R^{N})}

\newcommand{\x}{X^{s}(\R^{N+1}_{+})}
\newcommand{\y}{Y_{\mu}}

\newcommand{\La}{\Lambda}

\newcommand{\2}{2^{*}_{s}}

\newcommand{\ri}{\rightarrow}

\DeclareMathOperator{\dive}{div}
\DeclareMathOperator{\supp}{supp}
\DeclareMathOperator{\e}{\varepsilon}
\DeclareMathOperator{\p}{\epsilon}

\newtheorem{prop}{Proposition}[section]
\newtheorem{lem}{Lemma}[section]
\newtheorem{thm}{Theorem}[section]

\newtheorem{remark}{Remark}[section]

\title[A fractional relativistic Schr\"odinger equation with critical growth]{Concentration phenomena for a fractional relativistic Schr\"odinger equation with critical growth}

\author[V. Ambrosio]{Vincenzo Ambrosio}
\address{Vincenzo Ambrosio\hfill\break\indent
Dipartimento di Ingegneria Industriale e Scienze Matematiche \hfill\break\indent
Universit\`a Politecnica delle Marche\hfill\break\indent
Via Brecce Bianche, 12\hfill\break\indent
60131 Ancona (Italy)}
\email{v.ambrosio@staff.univpm.it}


\date{}

\begin{document}

\begin{abstract}
In this paper, we are concerned with the following fractional relativistic Schr\"odinger equation with critical growth:
\begin{equation*}
\left\{
\begin{array}{ll}
(-\Delta+m^{2})^{s}u + V(\e x) u= f(u)+u^{\2-1} \mbox{ in } \R^{N}, \\
u\in H^{s}(\R^{N}), \quad u>0 \, \mbox{ in } \R^{N},
\end{array}
\right.
\end{equation*}
where $\e>0$ is a small parameter, $s\in (0, 1)$, $m>0$, $N> 2s$, $\2=\frac{2N}{N-2s}$ is the fractional critical exponent, $(-\Delta+m^{2})^{s}$ is the fractional relativistic Schr\"odinger operator, $V:\R^{N}\ri \R$ is a continuous potential, and $f:\R\ri \R$ is a superlinear continuous nonlinearity with subcritical growth at infinity.  Under suitable assumptions on the potential $V$, we construct a family of positive solutions $u_{\e}\in H^{s}(\R^{N})$, with exponential decay, which concentrates around a local minimum of  $V$ as $\e\ri 0$.  
\end{abstract}
\maketitle

\section{Introduction}

\noindent
In this paper, we continue the study started in \cite{ADCDS} concerning the concentration phenomena for a class of fractional relativistic Schr\"odinger equations. More precisely, we focus on the following nonlinear fractional elliptic equation with critical growth:
\begin{equation}\label{P}
\left\{
\begin{array}{ll}
(-\Delta+m^{2})^{s}u + V(\e x) u= f(u)+u^{\2-1} \mbox{ in } \R^{N}, \\
u\in H^{s}(\R^{N}), \quad u>0 \mbox{ in } \R^{N},
\end{array}
\right.
\end{equation}
where $\e>0$ is a small parameter, $m>0$, $s\in (0, 1)$, $N> 2s$, $\2:=\frac{2N}{N-2s}$ is the fractional critical exponent, and $V:\R^{N}\ri \R$ and $f:\R\ri \R$ are continuous functions. 
The operator $(-\Delta+m^{2})^{s}$ is defined in Fourier space as multiplication by the symbol $(|k|^{2}+m^{2})^{s}$ (see \cite{Hormander, LL}), i.e., for each function $u:\R^{N}\ri \R$ that belongs to the Schwartz space $\mathcal{S}(\R^{N})$ of rapidly decreasing functions, we have
$$
\mathcal{F}((-\Delta+m^{2})^{s}u)(k):=(|k|^{2}+m^{2})^{s}\mathcal{F}u(k), \quad k\in \R^{N}, 
$$
where we denoted by
$$
\mathcal{F}u(k):=(2\pi)^{-\frac{N}{2}} \int_{\R^{N}} e^{-\imath k\cdot x} u(x)\, dx, \quad k\in \R^{N}, 
$$ 
the Fourier transform of $u$.
We also recall the following alternative representation of $(-\Delta+m^{2})^{s}$ in terms of singular integrals (see \cite{FF, LL}): 
\begin{align}\label{FFdef}
(-\Delta+m^{2})^{s}u(x):=m^{2s}u(x)+C(N,s) m^{\frac{N+2s}{2}} P.V. \int_{\R^{N}} \frac{u(x)-u(y)}{|x-y|^{\frac{N+2s}{2}}}K_{\frac{N+2s}{2}}(m|x-y|)\, dy, \quad x\in \R^{N},
\end{align}
where $P. V.$ indicates the Cauchy principal value, $K_{\nu}$ is the modified Bessel function of the third kind of index $\nu$ (see \cite{ArS, Erd}), 
and
$$
C(N, s):=2^{-\frac{N+2s}{2}+1} \pi^{-\frac{N}{2}} 2^{2s} \frac{s(1-s)}{\Gamma(2-s)}.
$$
When $s=\frac{1}{2}$, the operator $\sqrt{-\Delta+m^{2}}$ was considered in \cite{Weder1, Weder2} for spectral problems and has  a clear meaning in relativistic quantum mechanics. 
Indeed, the energy for the motion of a free relativistic particle of mass $m$ and momentum $p$ is given by:
$$
\mathcal{E}:=\sqrt{p^{2}c^{2}+m^{2}c^{4}},
$$ 
where $c$ is the speed of the light.
With the usual quantization rule $p\mapsto -\imath \hbar \nabla$,  where $\hbar$ is Planck's constant, we obtain the so-called relativistic Hamiltonian operator: 
$$
\mathcal{H}:=\sqrt{-\hbar^{2}c^{2}\Delta+m^{2}c^{4}}-mc^{2}.
$$
The point of the subtraction of the constant $mc^{2}$ is to make sure that the spectrum of the operator $\mathcal{H}$ is $[0, \infty)$, and this explains the terminology of relativistic Schr\"odinger operators for the operators of the form $\mathcal{H}+V(x)$, where $V(x)$ is a potential (see \cite{CMS}). 
Equations involving $\mathcal{H}$ arise in the study of time-dependent Schr\"odinger equations of the type:
\begin{equation*}
\imath \hbar\frac{\partial \Phi}{\partial t}=\mathcal{H}\Phi-f(x, |\Phi|^{2})\Phi, \quad  (t, x)\in \R\times \R^{N},
\end{equation*}
where $\Phi: \R\times \R^{N}\ri \mathbb{C}$ is a wave function and $f: \R^{N}\times [0, \infty)\ri \R$ is a nonlinear function, 
which describe the 
dynamics of systems consisting of identical spin-0 bosons whose motions are relativistic, for instance, boson stars.
Physical models related to $\mathcal{H}$ have been widely analyzed over the past 30 years, and there exists an important literature on the spectral properties of relativistic Hamiltonians; most of it has been strongly influenced by the works of Lieb on the stability of relativistic matter (see \cite{ES, FJL, LY1, LS} and references therein).
On the other hand, from a probabilistic point of view, $m^{2s}-(-\Delta+m^{2})^{s}$ is the infinitesimal generator of a L\'evy process $X^{2s, m}_{t}$ called $2s$-stable relativistic process having the following characteristic function:
$$
E^{0} e^{\imath k\cdot X^{2s, m}_{t}}=e^{-t[(|k|^{2}+m^{2})^{s}-m^{2s}]}, \quad k\in \R^{N};
$$
(see, for example, \cite{CMS, ryznar}). For a more detailed discussion on $(-\Delta+m^{2})^{s}$, we refer the interested reader to \cite{AJDE}.

When $m\ri 0$, $(-\Delta+m^{2})^{s}$ reduces to the well-known fractional Laplacian $(-\Delta)^{s}$ defined via Fourier transform by: 
$$
\mathcal{F}((-\Delta)^{s}u)(k):=|k|^{2s}\mathcal{F}u(k), \quad k\in \R^{N}, 
$$
or through singular integrals by:
\begin{align}\label{PIFFR}
(-\Delta)^{s}u(x): =C_{N, s} \, P.V. \int_{\R^{N}} \frac{u(x)-u(y)}{|x-y|^{N+2s}}\, dy, \quad x\in \R^{N}, \quad C_{N, s}:= \pi^{-\frac{N}{2}} 2^{2s} \frac{\Gamma(\frac{N+2s}{2})}{\Gamma(2-s)}s(1-s).
\end{align}
This operator has gained tremendous popularity during the last two decades thanks to its applications in different fields, such as, among others, phase transition phenomena, crystal dislocation, population dynamics, anomalous diffusion, flame propagation, chemical reactions of liquids, conservation laws, quasi-geostrophic flows, and water waves. 
Moreover, the fractional Laplacian is the infinitesimal generator of a (rotationally) symmetric $2s$-stable L\'evy process. For a very nice introduction to $(-\Delta)^{s}$ and its applications, consult \cite{BuVa, DPV}.
Note that the most striking difference between the operators $(-\Delta)^{s}$ and $(-\Delta+m^{2})^{s}$ is that the first one
is homogeneous in scaling, whereas the second one is inhomogeneous as should be clear from the presence of the Bessel function $K_{\nu}$ in \eqref{FFdef}.

We emphasize that in these years, several authors dealt with the existence and multiplicity of solutions for the following fractional  Schr\"odinger equation:
\begin{equation}\label{ANNALISA}
\left\{
\begin{array}{ll}
\e^{2s}(-\Delta)^{s}u + V(x) u= f(u)+\gamma |u|^{\2-2}u \mbox{ in } \R^{N}, \\
u\in H^{s}(\R^{N}), \quad u>0 \mbox{ in } \R^{N},
\end{array}
\right.
\end{equation}
where $\e>0$, $\gamma\in \{0, 1\}$, $V:\R^{N}\ri \R$ and $f:\R\ri \R$ satisfy suitable conditions (see, for instance, \cite{AmbrosioBOOK} and references therein).

As $s\ri 1$, equation \eqref{ANNALISA} boils down to the classical nonlinear Schr\"odinger equation of the form:
\begin{equation}\label{NONLINEARLIBRO}
\left\{
\begin{array}{ll}
-\e^{2} \Delta u+V(x)u=g(u) \mbox{ in } \R^{N}, \\
u\in H^{1}(\R^{N}), \quad u>0 \mbox{ in } \R^{N}.
\end{array}
\right.
\end{equation}
Since we cannot review the huge bibliography on this topic, we refer to \cite{ADOS, AF, DF, FigFu, FW, Rab, Wang} for some results on  the existence, multiplicity, and concentration of positive solutions to \eqref{NONLINEARLIBRO} for small $\e>0$.  
We recall that a positive solution $u_{\e}$ of \eqref{NONLINEARLIBRO} is said to concentrate at $x_{0}\in \R^{N}$ as $\e\ri 0$ if
\begin{align*}
\forall \delta>0, \quad \exists \e_{0}>0, R>0: \quad u_{\e}(x)\leq \delta, \quad \forall |x-x_{0}|\geq \e R, \, \e<\e_{0}.
\end{align*}
The interest in studying semiclassical solutions of \eqref{NONLINEARLIBRO}, i.e., solutions of \eqref{NONLINEARLIBRO} with small $\e>0$, 
is justified by the well-known fact that  the transition from quantum mechanics to classical mechanics can be described by letting $\e\ri 0$.
A typical feature of semiclassical solutions is that they tend to concentrate as $\e\ri 0$ around critical points of the potential $V$.

On the other hand, several existence and multiplicity results for fractional equations driven by $(-\Delta+m^{2})^{s}$, with $m>0$, have been established in \cite{Ajmp, ADCDS, BMP, CZN1, SecchiJDDE, Shen}. In particular, in \cite{ADCDS}, the author investigated \eqref{P} without the presence of the critical term $u^{\2-1}$ and obtained the existence of solutions concentrating in a given set of local minima of $V$ as $\e\ri 0$.
He also related the number of positive solutions to the topology of the set where $V$ attains its minimum value. 
We point out that, in all the aforementioned articles, only equations with subcritical nonlinearities are considered.    

Motivated by the previous facts, in this paper we examine the existence of concentrating solutions to \eqref{P}, by assuming that the potential $V:\R^{N}\ri \R$ is a continuous function fulfilling the following conditions:
\begin{compactenum}[$(V_1)$]
\item there exists $V_{1}\in (0, m^{2s})$ such that $\displaystyle{-V_{1}:=\inf_{x\in \R^{N}}V(x)}$,
\item there exists a bounded open set $\Lambda\subset \R^{N}$ such that
$$
-V_{0}:=\inf_{x\in \Lambda} V(x)<\min_{x\in \partial \Lambda} V(x),
$$
with $V_{0}>0$, and $0\in M:=\{x\in \Lambda: V(x)=-V_{0}\}$,
\end{compactenum}
and that the nonlinearity $f:\R\ri \R$ is a continuous, $f(t)=0$ for $t\leq 0$, and satisfies the following  hypotheses:
\begin{compactenum}[$(f_1)$]
\item $\lim_{t\ri 0} \frac{f(t)}{t}=0$,
\item there exist $p, q\in (2, \2)$ and $\lambda>0$ such that 
\begin{align*}
f(t)\geq \lambda t^{p-1}, \mbox{ for all } t\geq 0, \, \mbox{ and } \lim_{t\ri \infty} \frac{f(t)}{t^{q-1}}=0,
\end{align*}
where $\lambda>0$ is such that
\begin{itemize}
\item $\lambda>0$ if either $N\geq 4s$, or $2s<N<4s$ and $\2-2<p<\2$,
\item $\lambda>0$ is sufficiently large if $2s<N<4s$ and $2<p\leq \2-2$,
\end{itemize}
\item there exists $\theta\in (2, q)$ such that $0<\theta F(t)\leq t f(t)$ for all $t>0$, where $F(t):=\int_{0}^{t} f(\tau)\, d\tau$,
\item the function $t\mapsto \frac{f(t)}{t}$ is increasing in $(0, \infty)$.
\end{compactenum}
The main result  of this paper can be stated as follows:
\begin{thm}\label{thm1}
Assume that $(V_1)$-$(V_2)$ and $(f_1)$-$(f_4)$ hold.
Then, for every small $\e>0$, there exists a solution $u_{\e}$ to \eqref{P} such that $u_{\e}$ has a maximum point $x_{\e}$ satisfying
$$
\lim_{\e\ri 0} {\rm dist}(\e x_{\e}, M)=0,
$$
and for which
$$
0<u_{\e}(x)\leq C_{1}e^{-C_{2}|x-x_{\e}|}, \quad \mbox{ for all } x\in \R^{N},
$$
for suitable constants $C_{1}, C_{2}>0$. Moreover, for each sequence $(\e_{n})$ with $\e_{n}\ri 0$, there exists a subsequence, still denoted by itself, such that there exist a point $x_{0}\in M$ with $\e_{n}x_{\e_{n}}\ri x_{0}$ and a positive ground state solution $u\in \h$ of the limiting problem:  
$$
(-\Delta+m^{2})^{s}u-V_{0} u=f(u)+u^{\2-1} \quad \mbox{ in } \R^{N},
$$
for which we have
$$
u_{\e_{n}}(x)=u(x-x_{\e_{n}})+\mathcal{R}_{n}(x),
$$
where $\displaystyle{\lim_{n\ri \infty} \|\mathcal{R}_{n}\|_{\h}=0}$.
\end{thm}
The proof of Theorem \ref{thm1} relies on appropriate variational techniques. 
Since the operator $(-\Delta+m^{2})^{s}$ is nonlocal, we transform \eqref{P} into a degenerate elliptic equation in a half-space with a nonlinear Neumann boundary condition via a variant of the extension method \cite{CS} (see \cite{CZN1, FF, StingaT}).
Then, we adapt the penalization approach in \cite{DF}, the so-called {\it local mountain pass}, 
by building a convenient modification of the energy functional associated with the extended problem in such a way that the corresponding modified energy functional $J_{\e}$ satisfies the hypotheses of the mountain pass theorem \cite{AR}, and then we prove that, for $\e>0$ sufficiently small, the trace of the associated mountain pass solution is, indeed, a solution to the original equation with the stated properties. 
The modification of the functional corresponds to a penalization outside $\Lambda$, and this is why no other global assumptions are required. 
With respect to \cite{ADCDS}, it is more difficult to obtain compactness for $J_{\e}$ due to the presence of the critical exponent. 
To overcome this obstacle, we first estimate from above the mountain pass level $c_{\e}$ of $J_{\e}$, by constructing a suitable cut-off function.
Roughly speaking, we choose a function, appropriately rescaled, of the type $v_{\p}(x, y)=\vartheta(m y) \phi(x, y) w_{\p}(x, y)$, with $\p>0$, where $\vartheta$ is expressed via the Bessel function $K_{s}$, $\phi$ is a smooth cut-off function, and $w_{\p}$ is the $s$-harmonic extension of the extremal function $u_{\p}$ for the fractional Sobolev inequality (see \cite{CT}), in such a way that the control of the quadratic term 
\begin{align*}
\iint_{\R^{N+1}_{+}} y^{1-2s} (|\nabla v_{\p}|^{2}+m^{2}v^{2}_{\p})\, dx dy-m^{2s} \int_{\R^{N}} v_{\p}^{2}(x, 0)\, dx
\end{align*}
is, in some sense, reduced to the control of the term:
\begin{align*}
\iint_{\R^{N+1}_{+}} y^{1-2s} |\nabla (\phi w_{\p})|^{2}\, dx dy,
\end{align*}
and thus, we are able to verify that $c_{\e}<c_{*}$, where the threshold value $c_{*}$ depends on the best constant $S_{*}$ for the critical Sobolev trace inequality (see \cite{BRCDPS}), and the constants $V_{1}$, $m^{2s}$, and $\theta$ (see Lemma \ref{UPPERB}).
In view of this bound and by establishing a concentration-compactness principle in the spirit of Lions \cite{Lions1, Lions2}, we show that the modified energy functional  satisfies the Palais-Smale condition in the range $(0, c_{*})$ (see Lemmas \ref{CCL} and \ref{lemma2}). 
Finally, we prove that, for $\e>0$ small enough, the solution of the auxiliary problem is, indeed, solution of the original one by combining a Moser iteration scheme \cite{Moser}, a comparison argument, and some crucial properties of the Bessel kernel \cite{ArS, stein} (see Lemmas \ref{moser} and \ref{lem2.6AM}).
As far as we know, this is the first time that the penalization method is used to study the concentration phenomena for a fractional relativistic Schr\"odinger equation with critical growth.

The structure of the paper is the following. In section $2$, we define some function spaces. In section $3$, we focus on the modified problem. In section $4$, we deal with the autonomous critical problems related to the extended modified problem. Section $5$ is devoted to the proof of Theorem \ref{thm1}. In section $6$, we discuss a multiplicity result for \eqref{P}.

\subsection*{Notations:}
We denote the upper half-space in $\R^{N}$ by $\R^{N+1}_{+}:=\left\{(x, y)\in \R^{N+1}: y>0\right\}$. 
For $(x, y)\in \R^{N+1}_{+}$, we set $|(x, y)|:=\sqrt{|x|^{2}+y^{2}}$.
The letters $c$, $C$, $C'$, and $C_{i}$ will be repeatedly used to denote various positive constants whose exact values are irrelevant and can change from line to line. 
For $x\in \R^{N}$ and $R>0$, we will denote by $B_{R}(x)$ the ball in $\R^{N}$ centered at $x\in \R^{N}$ with radius $r>0$. When $x=0$, we set $B_{R}:=B_{R}(0)$.
For $(x_{0}, y_{0})\in \R^{N+1}_{+}$ and $R>0$, we put $B^{+}_{R}(x_{0}, y_{0}):=\{(x, y)\in \mathbb{R}^{N+1}_{+}: |(x,y)-(x_{0}, y_{0})|<R\}$ and 
$B^{+}_{R}:=B^{+}_{R}(0, 0)$. Let $p\in [1, \infty]$ and $A\subset \R^{N}$ be a measurable set.  
The notation $A^{c}:=\R^{N}\setminus A$ stands for the complement of $A$ in $\R^{N}$. 
We will use $|u|_{L^{p}(A)}$ for the $L^{p}$-norm of $u: \R^{N}\ri \R$. If $A=\R^{N}$, we simply write $|u|_{p}$ instead of $|u|_{L^{p}(\R^{N})}$. With $\|v\|_{L^{p}(\R^{N+1}_{+})}$ we denote the norm of $v\in L^{p}(\R^{N+1}_{+})$. For a generic real-valued function $w$, we set $w^{+}:=\max\{w, 0\}$ and $w^{-}:=\min\{w, 0\}$.

\section{Function spaces}

Let $H^{s}(\R^{N})$ be the fractional Sobolev space defined as the completion of $C^{\infty}_{c}(\R^{N})$ with respect to the norm
$$
|u|_{H^{s}(\R^{N})}:=\left( \int_{\R^{N}} (|k|^{2}+m^{2})^{s} |\mathcal{F}u(k)|^{2} dk \right)^{\frac{1}{2}}.
$$
Then, $H^{s}(\R^{N})$ is continuously embedded in $L^{p}(\R^{N})$ for all $p\in [2, \2]$ and compactly in  $L^{p}_{loc}(\R^{N})$ for all $p\in [1, \2)$; see \cite{Adams, AmbrosioBOOK, AJDE, DPV, LL}. 
We denote by $X^{s}(\R^{N+1}_{+})$ the completion of $C^{\infty}_{c}(\overline{\R^{N+1}_{+}})$ with respect to the norm:
$$
\|v\|_{X^{s}(\R^{N+1}_{+})}:=  \left(\iint_{\R^{N+1}_{+}} y^{1-2s} (|\nabla v|^{2}+m^{2}v^{2})\, dx dy  \right)^{\frac{1}{2}}.
$$
By \cite[Lemma 3.1]{FF} (see also \cite[Proposition 3.1.1]{DMV}), it follows that $\x$ is continuously embedded in $L^{2\gamma}(\R^{N+1}_{+}, y^{1-2s})$, i.e.,
\begin{equation}\label{weightedE}
\|v\|_{L^{2\gamma}(\R^{N+1}_{+}, y^{1-2s})}\leq \hat{S} \|v\|_{\x}, \quad \mbox{ for all } v\in \x,
\end{equation}
for some $\hat{S}>0$, where $\gamma:=1+\frac{2}{N-2s}$, and $L^{r}(\R^{N+1}_{+}, y^{1-2s})$ is the weighted Lebesgue space, with $r\in (1, \infty)$, endowed with the norm:
$$
\|v\|_{L^{r}(\R^{N+1}_{+}, y^{1-2s})}:=\left(\iint_{\R^{N+1}_{+}} y^{1-2s} |v|^{r}\, dx dy\right)^{\frac{1}{r}}.
$$
In light of \cite[Lemma 3.1.2]{DMV}, we also know that $\x$ is compactly embedded in $L^{2}(B_{R}^{+}, y^{1-2s})$ for all $R>0$. 
By \cite[Proposition 5]{FF}, there exists a (unique) linear trace operator ${\rm Tr}: \x\ri \h$ such that 
\begin{equation}\label{traceineq}
\sqrt{\sigma_{s}} |{\rm Tr}(v)|_{H^{s}(\R^{N})}\leq \|v\|_{\x}, \quad \mbox{ for all } v\in \x.
\end{equation}
where $\sigma_{s}:=2^{1-2s}\Gamma(1-s)/\Gamma(s)$; see \cite{BRCDPS, StingaT}.
In order to lighten the notation, we will denote ${\rm Tr}(v)$ by $v(\cdot, 0)$.
We note that \eqref{traceineq} yields
\begin{equation}\label{m-ineq}
\sigma_{s}m^{2s} \int_{\R^{N}} v^{2}(x, 0)\, dx\leq \iint_{\R^{N+1}_{+}} y^{1-2s} (|\nabla v|^{2}+m^{2}v^{2})\, dx dy,
\end{equation}
for all $v\in \x$, which can be also written as:
\begin{align}\label{Young}
\sigma_{s}\int_{\R^{N}} v^{2}(x, 0)\, dx\leq m^{-2s} \iint_{\R^{N+1}_{+}} y^{1-2s} |\nabla v|^{2}\, dx dy+m^{2-2s} \iint_{\R^{N+1}_{+}} y^{1-2s} v^{2}\, dx dy.
\end{align}
From the previous facts, we deduce the following fundamental embeddings.
\begin{thm}\label{Sembedding}
${\rm Tr}(X^{s}(\R^{N+1}_{+}))$ is continuously embedded in $L^{r}(\R^{N})$ for all $r\in [2, \2]$ and compactly embedded in $L^{r}_{loc}(\R^{N})$ for all $r\in [1, \2)$.
\end{thm}

\noindent
To deal with \eqref{P} via variational methods, we use a variant of the extension method \cite{CS} given in \cite{CZN1, FF, StingaT}. 
More precisely, for each $u\in \h$, there exists a unique function $U\in \x$ solving the  problem
\begin{align*}
\left\{
\begin{array}{ll}
-\dive(y^{1-2s} \nabla U)+m^{2}y^{1-2s}U=0 &\mbox{ in } \R^{N+1}_{+}, \\
U=u &\mbox{ on } \partial \R^{N+1}_{+}. 
\end{array}
\right.
\end{align*}
The function $U$ is called the extension of $u$ and fulfills the following properties:
\begin{compactenum}[$(E1)$]
\item
$$
\frac{\partial U}{\partial \nu^{1-2s}}:=-\lim_{y\ri 0} y^{1-2s} \frac{\partial U}{\partial y}(x,y)=\sigma_{s}(-\Delta+m^{2})^{s}u(x) \quad \mbox{ in } H^{-s}(\R^{N}),
$$
where $H^{-s}(\R^{N})$ denotes the dual of $H^{s}(\R^{N})$.
\item $\sqrt{\sigma_{s}}|u|_{\h}=\|U\|_{\x}\leq \|V\|_{\x}$ for all $V\in \x$ such that $V(\cdot,0)=u$,
\item if $u\in \mathcal{S}(\R^{N})$, then $U\in C^{\infty}(\R^{N+1}_{+})\cap C(\overline{\R^{N+1}_{+}})$, and it can be expressed as:
$$
U(x, y)=\int_{\R^{N}} P_{s, m}(x-z,y) u(z)\, dz,
$$
with
$$
P_{s, m}(x,y):=c'_{N,s} y^{2s} m^{\frac{N+2s}{2}} |(x, y)|^{-\frac{N+2s}{2}} K_{\frac{N+2s}{2}}(m|(x, y)|)
$$
and 
$$
c'_{N, s}:=p_{N, s}\frac{2^{\frac{N+2s}{2}-1}}{\Gamma(\frac{N+2s}{2})},
$$
where ${\displaystyle{p_{N, s}:=\pi^{-\frac{N}{2}}\frac{\Gamma(\frac{N+2s}{2})}{\Gamma(s)}}}$ is the constant for the (normalized) Poisson kernel with $m=0$  (see \cite{StingaT}).
\end{compactenum}

\begin{remark}
We recall (see \cite{FF}) that $P_{s, m}$ is the Fourier transform of $k\mapsto \vartheta(\sqrt{|k|^{2}+m^{2}})$ and that
\begin{align}\label{Nkernel}
\int_{\R^{N}} P_{s, m}(x, y)\, dx=\vartheta(m y),
\end{align}
where
\begin{align}\label{vartheta}
\vartheta(r):=\frac{2}{\Gamma(s)} \left(\frac{r}{2}\right)^{s}K_{s}(r)
\end{align}
belongs to $H^{1}(\R_{+}, y^{1-2s})$ and solves the following ordinary differential equation:
\begin{align}\label{pks}
\left\{
\begin{array}{ll}
\vartheta''+\frac{1-2s}{y} \vartheta'-\vartheta=0 \quad \mbox{ in } \R_{+}, \\
\vartheta(0)=1, \quad  \lim_{y\ri \infty} \vartheta(y)=0. 
\end{array}
\right.
\end{align}
We also have
\begin{align}\label{limks}
\int_{0}^{\infty} y^{1-2s}(|\vartheta'(y)|^{2}+|\vartheta(y)|^{2})\, dy=-\lim_{y\ri 0} y^{1-2s}\vartheta'(y)=\kappa_{s}.
\end{align}
\end{remark}

\noindent
Consequently, \eqref{P} can be realized in a local manner through the following nonlinear boundary value problem:
\begin{align}\label{EP}
\left\{
\begin{array}{ll}
-\dive(y^{1-2s} \nabla v)+m^{2}y^{1-2s}v=0 &\mbox{ in } \R^{N+1}_{+}, \\
\frac{\partial v}{\partial \nu^{1-2s}}=\sigma_{s} [-V_{\e} v(\cdot, 0)+f(v(\cdot, 0))+(v^{+}(\cdot, 0))^{\2-1}] &\mbox{ on } \R^{N}, 
\end{array}
\right.
\end{align}
where $V_{\e}(x):=V(\e x)$. For simplicity of notation, we will drop the constant $\sigma_{s}$ from the second equation in \eqref{EP}.
In order to examine \eqref{EP}, for $\e>0$, we introduce the space
$$
X_{\e}:=\left\{v\in X^{s}(\R^{N+1}_{+}): \int_{\R^{N}} V_{\e}(x) v^{2}(x, 0)\, dx<\infty\right\}
$$
equipped with the norm:
$$
\|v\|_{\e}:=\left(\|v\|^{2}_{\x}+ \int_{\R^{N}} V_{\e}(x) v^{2}(x, 0)\, dx \right)^{\frac{1}{2}}.
$$
Clearly, $X_{\e}\subset X^{s}(\R^{N+1}_{+})$, and using \eqref{m-ineq} and $(V_1)$, we see that
\begin{equation}\label{equivalent}
\|v\|_{X^{s}(\R^{N+1}_{+})}^{2}\leq \left(\frac{m^{2s}}{m^{2s}-V_{1}}\right) \|v\|^{2}_{\e}, \quad \mbox{ for all } v\in X_{\e}.
\end{equation}
Furthermore, $X_{\e}$ is a Hilbert space endowed with the inner product:
\begin{align*}
\langle v, w\rangle_{\e}:=\iint_{\R^{N+1}_{+}} y^{1-2s}(\nabla v\cdot\nabla w+m^{2}vw)\, dx dy+\int_{\R^{N}} V_{\e}(x) v(x, 0)w(x, 0)\, dx, \quad \mbox{ for all } v, w\in X_{\e}.
\end{align*}
Henceforth, with $X_{\e}^{*}$, we will denote the dual space of $X_{\e}$.

\section{Penalization argument}  
To study \eqref{EP}, we adapt the penalization approach in \cite{DF} (see also \cite{ADCDS}). 
Fix $\kappa>\max\{\frac{V_{1}}{m^{2s}-V_{1}}, \frac{\theta}{\theta-2}\}> 1$ and $a>0$ such that $f(a)+a^{\2-1}=\frac{V_{1}}{\kappa}a$. Define 
\begin{equation*}
\tilde{f}(t):=\left\{
\begin{array}{ll}
f(t)+(t^{+})^{\2-1}, &\mbox{ for } t<a, \\
\frac{V_{1}}{\kappa} t,  &\mbox{ for } t\geq a,
\end{array}
\right.
\end{equation*}
and
$$
g(x, t):=\chi_{\Lambda}(x) f(t)+(1-\chi_{\Lambda}(x)) \tilde{f}(t), \quad \mbox{ for } (x,t)\in \R^{N}\times \R,
$$
where $\chi_{\Lambda}$ denotes the characteristic function of $\Lambda$. Set $G(x,t):=\int_{0}^{t} g(x, \tau)\, d\tau$.
By assumptions $(f_1)$-$(f_4)$, it is easy to prove that $g$ is a Carath\'eodory function satisfying the following properties:
\begin{compactenum}[$(g_1)$]
\item $\lim_{t\ri 0} \frac{g(x, t)}{t}=0$ uniformly in $x\in \R^{N}$,
\item $g(x, t)\leq f(t)+t^{\2-1}$ for all $x\in \R^{N}$, $t>0$,
\item $(i)$ $0< \theta G(x,t)\leq tg(x, t)$ for all $x\in \Lambda$ and $t>0$, or, $x\in \Lambda^{c}$ and $0<t\leq a$, \\
$(ii)$ $0\leq 2 G(x,t)\leq tg(x, t)\leq \frac{V_{1}}{\kappa} t^{2}$ for all $x\in \Lambda^{c}$ and $t>0$,
\item for each $x\in \Lambda$, the function $t\mapsto \frac{g(x, t)}{t}$ is increasing in $(0, \infty)$, and for each $x\in \Lambda^{c}$, the function $t\mapsto \frac{g(x, t)}{t}$ is increasing in $(0, a)$.
\end{compactenum}

\noindent
Let us introduce the following auxiliary problem:
\begin{align}\label{MEP}
\left\{
\begin{array}{ll}
-\dive(y^{1-2s} \nabla v)+m^{2}y^{1-2s}v=0 &\mbox{ in } \R^{N+1}_{+}, \\
\frac{\partial v}{\partial \nu^{1-2s}}=-V_{\e}  v(\cdot, 0)+g_{\e}(\cdot, v(\cdot, 0)) &\mbox{ on } \R^{N},
\end{array}
\right.
\end{align}
where $g_{\e}(x,t):=g(\e x, t)$. It is clear that if $v_{\e}$ is a positive solution of \eqref{MEP} satisfying $v_{\e}(x,0)< a$ for all $x\in \Lambda_{\e}^{c}$, where $\Lambda_{\e}:=\{x\in \R^{N}: \e x\in \Lambda\}$, then $v_{\e}$ is a positive solution of \eqref{EP}.
The energy functional associated with \eqref{MEP}  is defined by:
$$
J_{\e}(v):=\frac{1}{2}\|v\|^{2}_{\e}-\int_{\R^{N}} G_{\e}(x,v(x,0))\, dx, \quad \mbox{ for all } v\in X_{\e}.
$$
It is standard to check that $J_{\e}\in C^{1}(X_{\e}, \R)$ and that its differential is given by
\begin{align*}
\langle J'_{\e}(v), w\rangle=\langle v, w\rangle_{\e}-\int_{\R^{N}} g_{\e}(x, v(x, 0)) w(x, 0)\, dx \quad \mbox{ for all } v, w\in X_{\e}.
\end{align*}
Hence, the critical points of $J_{\e}$ correspond to the weak solutions of \eqref{MEP}. To seek these critical
points, we will apply suitable variational arguments. First, we show that $J_{\e}$ possesses the geometric assumptions of the mountain pass theorem \cite{AR}.
\begin{lem}\label{lemma1}
The functional $J_{\e}$ satisfies the following properties:
\begin{compactenum}[$(i)$]
\item $J_{\e}(0)=0$,
\item there exist $\alpha, \rho>0$ such that $J_{\e}(v)\geq \alpha$ for all $v\in X_{\e}$ such that $\|v\|_{\e}=\rho$,
\item there exists $\bar{v}\in X_{\e}$ such that $\|\bar{v}\|_{\e}>\rho$ and $J_{\e}(\bar{v})<0$.
\end{compactenum}
\end{lem}
\begin{proof}
Condition $(i)$ is obvious. By $(f_1)$, $(f_2)$, $(g_1)$, and $(g_2)$, we see that for all $\eta>0$, there exists $C_{\eta}>0$ such that
\begin{equation}\label{growthg}
|g_{\e}(x, t)|\leq \eta |t|+C_{\eta}|t|^{\2-1}, \quad \mbox{ for } (x, t)\in \R^{N}\times \R,
\end{equation}
and
\begin{equation}\label{growthG}
|G_{\e}(x, t)|\leq \frac{\eta}{2} |t|^{2}+\frac{C_{\eta}}{\2}|t|^{\2}, \quad \mbox{ for } (x, t)\in \R^{N}\times \R.
\end{equation}
Pick $\eta\in (0, m^{2s}-V_{1})$. By \eqref{growthG}, \eqref{m-ineq}, \eqref{equivalent}, and using Theorem \ref{Sembedding}, we have
\begin{align*}
J_{\e}(v)&\geq \frac{1}{2} \|v\|^{2}_{\e}-\frac{\eta}{2} |v(\cdot, 0)|_{2}^{2}-\frac{C_{\eta}}{\2}  |v(\cdot, 0)|^{\2}_{\2} \\
&=\frac{1}{2} \|v\|^{2}_{\e}-\frac{\eta}{2m^{2s}} m^{2s} |v(\cdot, 0)|_{2}^{2}-\frac{C_{\eta}}{\2}  |v(\cdot, 0)|^{\2}_{\2} \\
&\geq \frac{1}{2}\|v\|^{2}_{\e}-\frac{\eta}{2m^{2s}} \|v\|^{2}_{\x}-C_{\eta}C\|v\|^{\2}_{\x} \\
&\geq \left(\frac{1}{2}-  \frac{\eta}{2(m^{2s}-V_{1})} \right) \|v\|^{2}_{\e}-C_{\eta}C' \|v\|^{\2}_{\e},
\end{align*}
from which we deduce that $(ii)$ is fulfilled. Finally, 
take $v_{0}\in C^{\infty}_{c}(\overline{\R^{N+1}_{+}})$ such that $v_{0}\geq 0$, $v_{0}\not\equiv 0$, and $\supp(v_{0}(\cdot, 0))\subset \Lambda_{\e}$. Then, for all $t>0$,
\begin{align*}
J_{\e}(tv_{0})&\leq \frac{t^{2}}{2} \|v_{0}\|^{2}_{\e}-\int_{\R^{N}} F(t v_{0}(x, 0))\, dx\\
&\leq \frac{t^{2}}{2} \|v_{0}\|^{2}_{\e}-C_{1} t^{\theta}\int_{\Lambda_{\e}} (v_{0}(x, 0))^{\theta}\, dx+C_{2} |\Lambda_{\e}|,
\end{align*}
where we have used $(g_{3})$. Since $\theta\in (2, \2)$, we obtain $J_{\e}(tv_{0})\ri -\infty$ as $t\ri \infty$.
\end{proof}

By Lemma \ref{lemma1} and invoking a variant of the mountain pass theorem without the Palais-Smale condition (see \cite[Theorem 2.9]{W}),
we can find a Palais-Smale sequence $(v_{n})\subset X_{\e}$ at the mountain pass level $c_{\e}$, i.e.,
\begin{align*}
J_{\e}(v_{n})\ri c_{\e} \quad \mbox{ and } \quad J'_{\e}(v_{n})\ri 0 \mbox{ in } X^{*}_{\e},
\end{align*}
as $n\ri \infty$, where
$$
c_{\e}:=\inf_{\gamma\in \Gamma_{\e}}\max_{t\in [0, 1]}J_{\e}(\gamma(t))
$$
and
$$
\Gamma_{\e}:=\{\gamma\in C([0, 1], X_{\e}): \gamma(0)=0, \, J_{\e}(\gamma(1))< 0\}.
$$
In view of the properties of $g$, it is easy to verify (see \cite{Rab, W}) that 
$$
c_{\e}=\inf_{v\in \mathcal{N}_{\e}} J_{\e}(v)=\inf_{v\in X_{\e}\setminus\{0\}} \max_{t\geq 0}J_{\e}(t v),
$$ 
where
$$
\mathcal{N}_{\e}:=\{v\in X_{\e}: \langle J'_{\e}(v), v\rangle=0\}
$$
is the Nehari manifold associated with $J_{\e}$.

\noindent
Next, we provide an important upper bound for the minimax level $c_{\e}$. For this purpose, we remember the following trace inequality (see \cite[Theorem 2.1]{BRCDPS}):
\begin{align}\label{BRANDLE}
\iint_{\R^{N+1}_{+}} y^{1-2s} |\nabla v|^{2}\, dx dy\geq S_{*} \left(\int_{\R^{N}} |v(x, 0)|^{\2}\, dx \right)^{\frac{2}{\2}},
\end{align}
for all $v\in X^{s}_{0}(\R^{N+1}_{+})$, where $X^{s}_{0}(\R^{N+1}_{+})$ is the completion of $C^{\infty}_{c}(\overline{\R^{N+1}_{+}})$ under the norm:
$$
\left(\iint_{\R^{N+1}_{+}}  y^{1-2s} |\nabla v|^{2}\, dx dy\right)^{\frac{1}{2}},
$$
and the exact value of the best constant $S_{*}=S(N, s)>0$ is
\begin{align*}
S_{*}:=\frac{2\pi^{s} \Gamma(1-s) \Gamma(\frac{N+2s}{2}) \Gamma(\frac{N}{2})^{\frac{2s}{N}}}{\Gamma(s) \Gamma(\frac{N-2s}{2}) \Gamma(N)^{\frac{2s}{N}}}.
\end{align*}
This constant is achieved on the family of functions $w_{\p}=\mathcal{E}_{s}(u_{\p})$, where $\mathcal{E}_{s}$ denotes the $s$-harmonic extension \cite{CS}, and
$$
u_{\p}(x):=\frac{\p^{\frac{N-2s}{2}}}{(|x|^{2}+\p^{2})^{\frac{N-2s}{2}}}, \quad \p>0;
$$  
see \cite{BRCDPS, CT, SV} for more details.
Hence,
\begin{align*}
w_{\p}(x, y):=(P_{s}(\cdot, y)*u_{\p})(x)=\int_{\R^{N}} P_{s}(x-\xi, y)u_{\p}(\xi)\, d\xi,
\end{align*}
where 
\begin{align*}
P_{s}(x, y):=\frac{p_{N, s}\, y^{2s}}{(|x|^{2}+y^{2})^{\frac{N+2s}{2}}}
\end{align*}
is the Poisson kernel for the extension problem in the half-space $\R^{N+1}_{+}$.
Note that $w_{\p}(x, y)=\p^{\frac{2s-N}{2}} w_{1}(\frac{x}{\p}, \frac{y}{\p})$.
\begin{lem}\label{UPPERB}
It holds $0<c_{\e}<\frac{s}{N}(\zeta S_{*})^{\frac{N}{2s}}$, where $\zeta:=1-\frac{V_{1}}{m^{2s}}\left(1+\frac{1}{\kappa}\right)\in (0, 1)$.
\end{lem}
\begin{proof}
For simplicity, we assume that $\e=1$. 
Let $\phi(x, y):=\phi_{0}(|(x, y)|)$, where $\phi_{0}\in C^{\infty}([0, \infty))$ is a non-increasing function such that
\begin{align*}
\phi_{0}(t)=1 \, \mbox{ if } t\in [0, 1], \quad \phi_{0}(t)=0 \,\mbox{ if } t\geq 2,
\end{align*}
and suppose that $B_{2}\subset \Lambda$.
Then, we consider 
\begin{align*}
\eta_{\p}(x, y):=\frac{(\phi w_{\p})(x, y)}{|(\phi w_{\p})(\cdot, 0)|_{\2}}, \quad \p>0.
\end{align*}
Let us recall the following useful estimates (see \cite{AmbrosioBOOK, BACDPS, SV}):
\begin{align}\label{HZ1}
\iint_{\R^{N+1}_{+}} y^{1-2s}|\nabla \eta_{\p}(x, y)|^{2}\, dx dy=S_{*}+O(\p^{N-2s}),
\end{align}
\begin{equation}\label{HZ2}
|\eta_{\p}(\cdot, 0)|_{2}^{2}=
\left\{
\begin{array}{ll}
O(\p^{2s}), &\mbox{ if } N>4s, \\
O(\p^{2s}\log(1/\p)), &\mbox{ if } N=4s,\\
O(\p^{N-2s}), &\mbox{ if } 2s<N<4s, \\
\end{array}
\right.
\end{equation}
\begin{equation}\label{HZ3}
|\eta_{\p}(\cdot, 0)|_{q}^{q}\geq
\left\{
\begin{array}{ll}
O(\p^{\frac{2N-(N-2s)q}{2}}), &\mbox{ if } q>\frac{N}{N-2s}, \\
O(\p^{\frac{N}{2}}|\log(\p)|), &\mbox{ if } q=\frac{N}{N-2s}, \\
O(\p^{\frac{(N-2s)q}{2}}), &\mbox{ if } q<\frac{N}{N-2s}.
\end{array}
\right.
\end{equation}
Now, we define $v_{\p}(x, y):=\vartheta(my) \eta_{\p, \beta}(x, y)$, where $\eta_{\p, \beta}(x, y)=\eta_{\p}(\beta x, \beta y)$, $\beta:=\zeta^{-\frac{2}{N-2s}}$, and $\vartheta$ is given in \eqref{vartheta}. Evidently, $v_{\p}(x, 0)=\eta_{\p, \beta}(x, 0)$ and
\begin{align*}
\partial_{x_{j}} v_{\p}(x, y)=\beta \vartheta(m y) \partial_{x_{j}} \eta_{\p, \beta}(x, y), \,\,  \mbox{ for all } j=1, \dots, N,
\end{align*}
\begin{align*}
\partial_{y} v_{\p}(x, y)=m\vartheta'(m y) \eta_{\p, \beta}(x, y)+\beta \vartheta(m y) \partial_{y} \eta_{\p, \beta}(x, y). 
\end{align*}
Therefore,
\begin{align}\label{J1}
\|v_{\p}\|^{2}_{X^{s}(\R^{N+1}_{+})}&= \iint_{\R^{N+1}_{+}} y^{1-2s} (|\nabla v_{\p}|^{2}+m^{2}v^{2}_{\p})\, dx dy \nonumber\\
&=\iint_{\R^{N+1}_{+}} y^{1-2s} \left[ \vartheta^{2}(m y) (\beta^{2} |\nabla \eta_{\p, \beta}|^{2}+m^{2}\eta^{2}_{\p, \beta})+2m\beta \vartheta'(m y) \vartheta(m y) \eta_{\p, \beta} \partial_{y} \eta_{\p, \beta}+m^{2}(\vartheta'(m y))^{2}\eta^{2}_{\p, \beta} \right] \, dx dy \nonumber\\
&=\iint_{\R^{N+1}_{+}} y^{1-2s} \vartheta^{2}(m y)\beta^{2} |\nabla \eta_{\p, \beta}|^{2}\, dx dy+\iint_{\R^{N+1}_{+}} 2m\beta y^{1-2s}\vartheta'(m y) \vartheta(m y) \eta_{\p, \beta} \partial_{y} \eta_{\p, \beta}\, dx dy \nonumber\\
&\quad +\iint_{\R^{N+1}_{+}} y^{1-2s}  m^{2} \left[(\vartheta'(m y))^{2}+(\vartheta(m y))^{2}\right]\eta^{2}_{\p, \beta}  \, dx dy \nonumber\\
&=\iint_{\R^{N+1}_{+}} y^{1-2s} \vartheta^{2}(m y)\beta^{2} |\nabla \eta_{\p, \beta}|^{2}\, dx dy+m^{2s}\int_{\R^{N}} \eta^{2}_{\p, \beta}(x, 0)\, dx,
\end{align} 
where we have used integration by parts, $\kappa_{s}=1$, \eqref{pks}, and \eqref{limks} to deduce that
\begin{align*}
&\iint_{\R^{N+1}_{+}} 2 m\beta y^{1-2s}\vartheta'(m y) \vartheta(m y) \eta_{\p, \beta} \partial_{y} \eta_{\p, \beta}\, dx dy=\iint_{\R^{N+1}_{+}} m y^{1-2s}\vartheta'(m y) \vartheta(m y) \partial_{y}(\eta^{2}_{\p, \beta})\, dx dy \\
&=-\int_{\R^{N}} \left[ \lim_{y\ri 0} \vartheta(my) \vartheta'(my) my^{1-2s} \eta^{2}_{\p, \beta}(x, y) \right]\, dx -\iint_{\R^{N+1}_{+}}  m \eta^{2}_{\p, \beta} \partial_{y}[y^{1-2s} \vartheta(m y) \vartheta'(my)]\, dx dy\\
&=m^{2s}\kappa_{s} \int_{\R^{N}} \eta^{2}_{\p, \beta}(x, 0)\, dx-\iint_{\R^{N+1}_{+}} y^{1-2s}  m^{2} \left[(\vartheta'(m y))^{2}+(\vartheta(m y))^{2}\right]\eta^{2}_{\p, \beta}  \, dx dy \\
&=m^{2s} \int_{\R^{N}} \eta^{2}_{\p, \beta}(x, 0)\, dx-\iint_{\R^{N+1}_{+}} y^{1-2s}  m^{2} \left[(\vartheta'(m y))^{2}+(\vartheta(m y))^{2}\right]\eta^{2}_{\p, \beta}  \, dx dy.
\end{align*}
Consequently, using \eqref{J1}, a change of variable theorem, $|\eta_{\e}(\cdot, 0)|_{\2}=1$, $(g_{2})$, and $0\leq \vartheta(my)\leq 1$, for all $t>0$,
\begin{align}\label{DEDDY}
J_{1}(t v_{\p})&\leq \frac{t^{2}}{2} \|v_{\p}\|^{2}_{X^{s}(\R^{N+1}_{+})}+\frac{1}{2}\int_{\R^{N}} V(x) (tv_{\p}(x, 0))^{2}\, dx-\frac{\lambda t^{p}}{p} \int_{\R^{N}} |v_{\p}(x, 0)|^{p}\, dx -\frac{t^{\2}}{\2}   \int_{\R^{N}} |v_{\p}(x, 0)|^{\2}\, dx \nonumber \\
&\leq \frac{t^{2}}{2} \iint_{\R^{N+1}_{+}} y^{1-2s} \beta^{2} |\nabla \eta_{\p, \beta}|^{2}\, dx dy+(\bar{V}+m^{2s})\frac{t^{2}}{2}|\eta_{\p, \beta}(\cdot, 0)|^{2}_{2}-\frac{\lambda t^{p}}{p} |\eta_{\p, \beta}(\cdot, 0)|^{p}_{p}-\frac{t^{\2}}{\2}|\eta_{\p, \beta}(\cdot, 0)|_{\2}^{\2} \nonumber \\
&\leq \frac{t^{2}}{2} \left[ \beta^{2s-N}\iint_{\R^{N+1}_{+}} y^{1-2s}  |\nabla \eta_{\p}|^{2}\, dx dy+\beta^{-N}(\bar{V}+m^{2s})|\eta_{\p}(\cdot, 0)|^{2}_{2}\right]-\frac{\lambda t^{p}}{p}\beta^{-N} |\eta_{\p}(\cdot, 0)|^{p}_{p}-\frac{t^{\2}}{\2}\beta^{-N}=: \alpha(t), 
\end{align}
where $\bar{V}:=\max_{\bar{\Lambda}} V>-m^{2s}$. Clearly, $\alpha(t)>0$ for $t>0$ small and $\alpha(t)\ri -\infty$ as $t\ri \infty$, so $\alpha(t)$ attains  its maximum at some $t_{\p}>0$ with $\alpha'(t_{\p})=0$. 
This fact combined with\eqref{HZ1}-\eqref{HZ3} implies that there exist $\delta_{1}, \delta_{2}>0$, independent of $\p>0$, such that
\begin{align*}
\delta_{1}\leq t_{\p}\leq \delta_{2}.
\end{align*}
Then, observing that the function
\begin{align*}
t\mapsto \frac{t^{2}}{2}\left[ \beta^{2s-N}\iint_{\R^{N+1}_{+}} y^{1-2s}  |\nabla \eta_{\p}|^{2}\, dx dy+\beta^{-N}(\bar{V}+m^{2s})|\eta_{\p}(\cdot, 0)|^{2}_{2}\right]-\frac{t^{\2}}{\2}\beta^{-N}
\end{align*}
is increasing in the interval
\begin{align*}
\Biggl[0, \left(\beta^{2s-N}\iint_{\R^{N+1}_{+}} y^{1-2s}  |\nabla \eta_{\p}|^{2}\, dx dy+\beta^{-N}(\bar{V}+m^{2s})|\eta_{\p}(\cdot, 0)|^{2}_{2}\right)^{\frac{1}{\2-2}} \beta^{\frac{N}{\2-2}} \Biggr],
\end{align*}
and exploiting the well-known inequality
\begin{align*}
(a+b)^{\alpha}\leq a^{\alpha}+\alpha(a+b)^{\alpha-1}, \quad \mbox{ for } a, b>0, \, \alpha\geq 1,
\end{align*}
and the estimates \eqref{HZ1}-\eqref{HZ3}, we can deduce that
\begin{align*}
\alpha(t_{\p})&\leq \frac{s}{N} \left(\beta^{2s-N}\iint_{\R^{N+1}_{+}} y^{1-2s}  |\nabla \eta_{\p}|^{2}\, dx dy+\beta^{-N}(\bar{V}+m^{2s})|\eta_{\p}(\cdot, 0)|^{2}_{2}  \right)^{\frac{N}{2s}}\beta^{\frac{N(N-2s)}{4s}}-\lambda C|\eta_{\p}(\cdot, 0)|_{p}^{p} \\
&\leq \frac{s}{N}(\zeta S_{*})^{\frac{N}{2s}}+O(\p^{N-2s})+C|\eta_{\p}(\cdot, 0)|^{2}_{2}-\lambda C|\eta_{\p}(\cdot, 0)|_{p}^{p},
\end{align*}
where we have used 
\begin{align*}
\beta^{(2s-N)\frac{N}{2s}+\frac{N(N-2s)}{4s}}=\beta^{-\frac{N(N-2s)}{4s}}=\left(\zeta^{-\frac{2}{N-2s}}\right)^{-\frac{N(N-2s)}{4s}}=\zeta^{\frac{N}{2s}}.
\end{align*}
Next, we show that, for $\p>0$ small enough,
\begin{align}\label{ALEX}
\alpha(t_{\p})<\frac{s}{N}(\zeta S_{*})^{\frac{N}{2s}}.
\end{align} 
Assume that $N>4s$. Thus, $p>2>\frac{N}{N-2s}$, and using \eqref{HZ3}, we find
\begin{align*}
\alpha(t_{\p})\leq \frac{s}{N}(\zeta  S_{*})^{\frac{N}{2s}}+O(\p^{N-2s})+O(\p^{2s})-O(\p^{N-\frac{(N-2s)p}{2}}).
\end{align*}
Since
\begin{align*}
N-\frac{(N-2s)p}{2}<2s<N-2s,
\end{align*}
we infer that \eqref{ALEX} holds 
as long as $\p>0$ is sufficiently small.

If $N=4s$, then $p>2=\frac{N}{N-2s}$, and in view of \eqref{HZ3}, we obtain  
\begin{align*}
\alpha(t_{\p})\leq \frac{s}{N}(\zeta  S_{*})^{\frac{N}{2s}}+O(\p^{2s}(1+|\log \p|))-O(\p^{4s-sp}).
\end{align*}
Observing that $p>2$ yields 
\begin{align*}
\lim_{\p\ri 0} \frac{\p^{4s-sp}}{\p^{2s}(1+|\log \p|)}=\infty,
\end{align*}
we arrive at the assertion for $\p>0$ small enough.

Now, let us consider the case $2s<N<4s$. First, we suppose that $\2-2<p<\2$. Hence,
\begin{align*}
p>\2-2>\frac{N}{N-2s},
\end{align*}
which combined with \eqref{HZ3} gives 
\begin{align*}
\alpha(t_{\p})\leq \frac{s}{N}(\zeta  S_{*})^{\frac{N}{2s}}+O(\p^{N-2s})-O(\p^{N-\frac{(N-2s)p}{2}}).
\end{align*}
Because
\begin{align*}
N-\frac{(N-2s)p}{2}<N-2s<2s,
\end{align*}
we conclude that \eqref{ALEX} is satisfied for $\p>0$ small enough.
Second, we deal with the case $2<p\leq \2-2$. We distinguish the following subcases:
\begin{align*}
2<p<\frac{N}{N-2s}, \quad p=\frac{N}{N-2s}, \quad \mbox{ and } \quad \frac{N}{N-2s}<p\leq \2-2.
\end{align*}
If $2<p<\frac{N}{N-2s}$, by \eqref{HZ3}, we see that
\begin{align*}
\alpha(t_{\p})\leq \frac{s}{N}(\zeta  S_{*})^{\frac{N}{2s}}+O(\p^{N-2s})-\lambda O(\p^{\frac{(N-2s)p}{2}}),
\end{align*}
and noting that 
\begin{align*}
N-2s<\frac{(N-2s)p}{2},
\end{align*}
we can put $\lambda=\p^{-\nu}$, with $\nu>\frac{(N-2s)(p-2)}{2}$, to reach the required estimate.

If $p=\frac{N}{N-2s}$, then, thanks to \eqref{HZ3}, we have
\begin{align*}
\alpha(t_{\p})\leq \frac{s}{N}(\zeta S_{*})^{\frac{N}{2s}}+O(\p^{N-2s})-\lambda O(\p^{\frac{N}{2}}|\log \p|),
\end{align*}
and taking $\lambda=\p^{-\nu}$, with $\nu>2s-\frac{N}{2}$, we deduce the assertion for $\p>0$ small enough.

Finally, when $\frac{N}{N-2s}<p\leq \2-2$, it follows from \eqref{HZ3} that
\begin{align*}
\alpha(t_{\p})\leq \frac{s}{N}(\zeta S_{*})^{\frac{N}{2s}}+O(\p^{N-2s})-\lambda O(\p^{N-\frac{(N-2s)p}{2}})  ,
\end{align*}
and choosing $\lambda=\p^{-\nu}$, with $\nu>2s-\frac{(N-2s)p}{2}$, we obtain the claim for $\p>0$ sufficiently small.\\
Consequently, \eqref{DEDDY} and \eqref{ALEX} yield
\begin{align*}
\max_{t\geq 0} J_{1}(t v_{\p})<\frac{s}{N} (\zeta S_{*})^{\frac{N}{2s}},
\end{align*}
which together with $c_{1}\leq \max_{t\geq 0} J_{1}(t v_{\p})$ implies the desired conclusion.
\end{proof}

\noindent
In what follows, we show that $J_{\e}$ satisfies a local compactness condition. First, we prove the boundedness of Palais-Smale sequences of $J_{\e}$. 
\begin{lem}\label{boundPS}
Let $0<c<\frac{s}{N}(\zeta S_{*})^{\frac{N}{2s}}$ and $(v_{n})\subset X_{\e}$ be a Palais-Smale sequence of $J_{\e}$ at the level $c$. Then, $(v_{n})$ is bounded in $X_{\e}$.
\end{lem}
\begin{proof}
By assumptions, we know that
\begin{align}\label{PSSQ}
J_{\e}(v_{n})\ri c \quad \mbox{ and } \quad J'_{\e}(v_{n})\ri 0 \mbox{ in } X^{*}_{\e},
\end{align}
as  $n\ri \infty$. 
Using \eqref{PSSQ}, $(g_3)$, \eqref{m-ineq}, and \eqref{equivalent}, we see that, for $n$ big enough, 
\begin{align*}
c+1+\|v_{n}\|_{\e}&\geq J_{\e}(v_{n})-\frac{1}{\theta}\langle J'_{\e}(v_{n}), v_{n}\rangle\\
&=\left(\frac{1}{2}-\frac{1}{\theta}\right) \|v_{n}\|^{2}_{\e}+\frac{1}{\theta} \int_{\R^{N}} g_{\e}(x, v_{n}(x, 0))v_{n}(x, 0)-\theta G_{\e}(x, v_{n}(x, 0))\, dx\\
&\geq \left(\frac{1}{2}-\frac{1}{\theta}\right) \|v_{n}\|^{2}_{\e}-\left(\frac{1}{2}-\frac{1}{\theta}\right) \frac{V_{1}}{\kappa} \int_{\R^{N}}  v_{n}^{2}(x, 0)\, dx \\
&= \left(\frac{1}{2}-\frac{1}{\theta}\right) \|v_{n}\|^{2}_{\e}-\left(\frac{1}{2}-\frac{1}{\theta}\right) \frac{V_{1}}{\kappa m^{2s}} m^{2s}\int_{\R^{N}}  v_{n}^{2}(x, 0)\, dx \\
&\geq \left(\frac{1}{2}-\frac{1}{\theta}\right) \|v_{n}\|^{2}_{\e}-\left(\frac{1}{2}-\frac{1}{\theta}\right) \frac{V_{1}}{\kappa m^{2s}}\|v_{n}\|^{2}_{\x} \\
&\geq \left(\frac{1}{2}-\frac{1}{\theta}\right) \left(1- \frac{V_{1}}{\kappa(m^{2s}-V_{1})}\right) \|v_{n}\|^{2}_{\e}.
\end{align*}
Since $\theta>2$ and $\kappa>\frac{V_{1}}{m^{2s}-V_{1}}$, we conclude that $(v_{n})$ is bounded in $X_{\e}$.  
\end{proof}

\begin{remark}\label{REMKINDERF}
If $0<c<\frac{s}{N}(\zeta S_{*})^{\frac{N}{2s}}$ and $(v_{n})\subset X_{\e}$ is a Palais-Smale sequence of $J_{\e}$ at the level $c$, then we may always assume that $(v_{n})$ is nonnegative. In fact, Lemma \ref{boundPS} implies that also $(v^{-}_{n})$ is bounded in $X_{\e}$. Then, we have $\langle J'_{\e}(v_{n}), v^{-}_{n}\rangle=o_{n}(1)$, which combined with $g(\cdot, t)=0$ for $t\leq 0$ yields $\|v^{-}_{n}\|_{\e}=o_{n}(1)$. Furthermore, it is easy to check that $J_{\e}(v_{n})=J_{\e}(v^{+}_{n})+o_{n}(1)$ and $J'_{\e}(v_{n})=J'_{\e}(v^{+}_{n})+o_{n}(1)$.
\end{remark}

\noindent
The next concentration-compactness principle in the spirit of Lions \cite{Lions1, Lions2} will be used in the proof of the local compactness of $J_{\e}$. We start by recalling some useful definitions. 
A sequence $(u_{n})\subset L^{1}(\R^{N})$ is tight if for every $\xi>0$, there exists $R>0$ such that
$$
\int_{B^{c}_{R}} |u_{n}|\, dx<\xi,  \quad \mbox{ for all } n\in \mathbb{N}.
$$
A sequence $(v_{n})\subset \x$ is tight if for every $\xi>0$, there exists $R>0$ such that 
$$
\iint_{\R^{N+1}_{+}\setminus B^{+}_{R}} y^{1-2s}(|\nabla v_{n}|^{2}+m^{2}v_{n}^{2})\, dx dy<\xi,  \quad \mbox{ for all } n\in \mathbb{N}.
$$
\begin{remark}
Let us observe that if $(v_{n})$ is a bounded tight sequence in $\x$, then $(|v_{n}(\cdot, 0)|^{\2})$ is a (bounded) tight sequence in $L^{1}(\R^{N})$. 
To prove this, let $R>0$ and consider $\eta_{R}\in C^{\infty}(\overline{\R^{N+1}_{+}})$ given by:
\begin{equation*}
\eta_{R}(x, y):=
\left\{
\begin{array}{ll}
0, &\mbox{ if } (x, y)\in B^{+}_{R/2}, \\
1, &\mbox{ if } (x, y)\in \overline{\R^{N+1}_{+}}\setminus B^{+}_{R},
\end{array}
\right.
\end{equation*}
with $0\leq \eta_{R}\leq 1$ and $\|\nabla \eta_{R}\|_{L^{\infty}(\R^{N+1}_{+})}\leq C/R$, for some $C>0$ independent of $R>0$. 
Using the definition and the properties of $\eta_{R}$, Theorem \ref{Sembedding}, the boundedness of $(v_{n})$ in $\x$, and $|x+y|^{2}\leq 2(|x|^{2}+|y|^{2})$ for all $x, y\in \R^{N}$, we have that, for all $n\in \mathbb{N}$,
\begin{align*}
&\left(\int_{B^{c}_{R}} |v_{n}(x, 0)|^{\2}\, dx  \right)^{\frac{2}{\2}} \leq \left(\int_{\R^{N}} |(v_{n}\eta_{R})(x, 0)|^{\2}\, dx  \right)^{\frac{2}{\2}} \\
&\leq C \left(\iint_{\R^{N+1}_{+}} y^{1-2s}(|\nabla(v_{n}\eta_{R})|^{2}+m^{2}v^{2}_{n}\eta_{R}^{2})\, dxdy\right)\\
&\leq C\left(\iint_{\R^{N+1}_{+}} y^{1-2s}(|\nabla v_{n}|^{2}+m^{2}v_{n}^{2})\eta^{2}_{R}\, dxdy+\iint_{\R^{N+1}_{+}} y^{1-2s}|\nabla \eta_{R}|^{2}v^{2}_{n}\, dxdy\right) \\
&\leq C\left(\iint_{\R^{N+1}_{+}\setminus B^{+}_{R/2}} y^{1-2s}(|\nabla v_{n}|^{2}+m^{2}v_{n}^{2})\, dxdy+\frac{C}{R^{2}}\right).
\end{align*}
From the aforementioned estimate and the tightness of $(v_{n})$ in $\x$, we derive that $(|v_{n}(\cdot, 0)|^{\2})$ is tight in $L^{1}(\R^{N})$, as desired. 
\end{remark}
\begin{prop}\label{CCL}
Let $(v_{n})$ be a bounded tight sequence in $\x$ such that $v_{n}\rightharpoonup v$ in $\x$. 
Let $\mu$ and $\nu$ be two bounded nonnegative measures on $\R^{N+1}_{+}$ and $\R^{N}$, respectively, and such that
\begin{align}
&y^{1-2s} (|\nabla v_{n}|^{2}+m^{2}v_{n}^{2})\rightharpoonup \mu \quad \mbox{ weakly in the sense of measures } \label{46FS1}
\end{align}
and
\begin{align}
&|v_{n}(\cdot, 0)|^{\2}\rightharpoonup \nu \quad \mbox{ weakly in the sense of measures. }  \label{46FS2}
\end{align}
Then, there exist an at most countable set $I$ and three families $(x_{i})_{i\in I}\subset \R^{N}$, $(\mu_{i})_{i\in I}\subset (0, \infty)$,  $(\nu_{i})_{i\in I}\subset (0, \infty)$ such that
\begin{align}
&\nu=|v(\cdot, 0)|^{\2}+\sum_{i\in I} \nu_{i} \delta_{x_{i}}, \label{47FS}\\
&\mu\geq y^{1-2s} (|\nabla v|^{2}+m^{2}v^{2})+\sum_{i\in I} \mu_{i} \delta_{(x_{i}, 0)}, \label{48FS} \\
&\mu_{i}\geq S_{*} \nu_{i}^{\frac{2}{2^{*}_{s}}}, \quad \mbox{ for all }  i\in I.  \label{49FS} 
\end{align}
\end{prop}
\begin{proof}
We follow the strategy used in the proof of \cite[Lemma I.1]{Lions1} (see also \cite[Lemma 2.3]{Lions2}, \cite[Proposition 3.2.1]{DMV}, and \cite[Theorem 5.1]{BACDPS}).
We first suppose that $v\equiv 0$. We claim that, for all $\varphi\in C^{\infty}_{c}(\overline{\R^{N+1}_{+}})$, it holds
\begin{equation}\label{DMV1}
\left(\int_{\R^{N}} |\varphi(x, 0)|^{\2} d\nu\right)^{\frac{1}{\2}}\leq C_{0}\left(\iint_{\R^{N+1}_{+}} \varphi^{2} \, d\mu\right)^{\frac{1}{2}},
\end{equation}
for some constant $C_{0}>0$. For this purpose, we fix  $\varphi\in C^{\infty}_{c}(\overline{\R^{N+1}_{+}})$ and let $K:=\supp(\varphi)$. 
By \eqref{BRANDLE}, we deduce that
\begin{equation}\label{DMV2}
\left(\int_{\R^{N}} |(\varphi v_{n})(x, 0)|^{\2} dx \right)^{\frac{1}{\2}}\leq S^{-\frac{1}{2}}_{*} \left(\iint_{\R^{N+1}_{+}} y^{1-2s} [|\nabla (\varphi v_{n})|^{2}+m^{2}(\varphi v_{n})^{2}] \, dx dy\right)^{\frac{1}{2}}.
\end{equation}
Now, we note that \eqref{46FS2} implies that
\begin{align}\label{DMV3}
\int_{\R^{N}} |(\varphi v_{n})(x, 0)|^{\2} dx \ri \int_{\R^{N}}  |\varphi(x, 0)|^{\2} d\nu.
\end{align}
On the other hand,
\begin{align}\label{DMV4}
&\iint_{\R^{N+1}_{+}} y^{1-2s} [|\nabla (\varphi v_{n})|^{2}+m^{2}(\varphi v_{n})^{2}] \, dx dy \nonumber \\
&=\iint_{\R^{N+1}_{+}} y^{1-2s} \varphi^{2}[|\nabla v_{n}|^{2}+m^{2}v_{n}^{2}] \, dx dy +\iint_{\R^{N+1}_{+}} y^{1-2s} v_{n}^{2} |\nabla \varphi|^{2} \, dx dy \nonumber \\
&\quad+2 \iint_{\R^{N+1}_{+}} y^{1-2s} v_{n}\varphi\nabla \varphi\cdot \nabla v_{n}\, dx dy. 
\end{align}
Since $H^{1}(K, y^{1-2s})$ is compactly embedded in $L^{2}(K, y^{1-2s})$ (see \cite[Lemma 3.1.2]{DMV}), we have that $v_{n}\ri 0$ in $L^{2}(K, y^{1-2s})$, which yields
\begin{equation}\label{DMV5}
\iint_{\R^{N+1}_{+}} y^{1-2s} v_{n}^{2} |\nabla \varphi|^{2} \, dx dy\leq C\iint_{K} y^{1-2s} v_{n}^{2}  \, dx dy \ri 0.
\end{equation}
Furthermore, the H\"older inequality, the boundedness of $(v_{n})$ in $\x$, and \eqref{DMV5} lead to
\begin{align}\label{DMV6}
&\left| \iint_{\R^{N+1}_{+}} y^{1-2s} v_{n}\varphi\nabla \varphi\cdot \nabla v_{n}\, dx dy\right|  \nonumber \\
&\leq \left(\iint_{\R^{N+1}_{+}} y^{1-2s} v^{2}_{n}|\nabla \varphi|^{2}\, dx dy\right)^{\frac{1}{2}} \left(\iint_{\R^{N+1}_{+}} y^{1-2s} \varphi^{2} |\nabla v_{n}|^{2}\, dx dy\right)^{\frac{1}{2}} \nonumber\\
&\leq C  \left(\iint_{\R^{N+1}_{+}} y^{1-2s} v^{2}_{n} \, dx dy\right)^{\frac{1}{2}} \left(\iint_{\R^{N+1}_{+}} y^{1-2s} |\nabla v_{n}|^{2}\, dx dy\right)^{\frac{1}{2}} \nonumber\\
&\leq C\left(\iint_{K} y^{1-2s} v_{n}^{2}\, dx dy\right)^{\frac{1}{2}}\ri 0.
\end{align}
Finally, taking \eqref{46FS1} into account, we see that 
\begin{equation}\label{DMV7}
\iint_{\R^{N+1}_{+}} y^{1-2s} \varphi^{2}(|\nabla v_{n}|^{2}+m^{2}v_{n}^{2}) \, dx dy\ri \iint_{\R^{N+1}_{+}}  \varphi^{2} \, d\mu.
\end{equation}
Putting together \eqref{DMV2}-\eqref{DMV7}, we can infer that \eqref{DMV1} holds with $C_{0}=S^{-\frac{1}{2}}_{*}$. 
Now, assume that $v$ is not necessarily $0$. Set $w_{n}:=v_{n}-v$. Clearly, $(w_{n})$ is a bounded tight sequence in $\x$ such that $w_{n}\rightharpoonup 0$ in $\x$. Moreover, there exist two bounded nonnegative measures $\tilde{\mu}$ and $\tilde{\nu}$ on $\R^{N+1}_{+}$ and $\R^{N}$, respectively, such that
\begin{align}
&y^{1-2s} (|\nabla w_{n}|^{2}+m^{2}w_{n}^{2})\rightharpoonup \tilde{\mu} \quad \mbox{ weakly in the sense of measures, } \label{46FS1w}\\
&|w_{n}(\cdot, 0)|^{\2}\rightharpoonup \tilde{\nu} \quad \mbox{ weakly in the sense of measures. } \label{46FS2w}
\end{align}
Then, we are in the previous case, and we can use \eqref{DMV1} to deduce that
\begin{equation*}
\left(\int_{\R^{N}} |\varphi(x, 0)|^{\2} d\tilde{\nu}\right)^{\frac{1}{\2}}\leq S^{-\frac{1}{2}}_{*} \left(\iint_{\R^{N+1}_{+}} \varphi^{2} \, d\tilde{\mu}\right)^{\frac{1}{2}}, \quad \mbox{ for all } \varphi\in C^{\infty}_{c}(\overline{\R^{N+1}_{+}}).
\end{equation*}
Hence, as in \cite[Lemma 1.2]{Lions1}, we can find an at most countable set $I$, a family of distinct points $(x_{i})_{i\in I}\subset \R^{N}$ and $(\nu_{i})_{i\in I}\subset (0, \infty)$ such that
\begin{align}\label{LGSONY}
\tilde{\nu}=\sum_{i\in I} \nu_{i} \delta_{x_{i}}.
\end{align}
Pick $\varphi\in C^{\infty}_{c}(\overline{\R^{N+1}_{+}})$. By the Brezis-Lieb lemma \cite{BL}, we know that
$$
|(\varphi w_{n})(\cdot, 0)|_{\2}^{\2}=|(\varphi v_{n})(\cdot, 0)|_{\2}^{\2}-|(\varphi v)(\cdot, 0)|_{\2}^{\2}+o_{n}(1).
$$
The aforementioned fact combined with \eqref{46FS2}, \eqref{46FS2w}, and \eqref{LGSONY}, the boundedness of $(v_{n})$ in $\x$, and the tightness of $(|v_{n}(\cdot, 0)|^{\2})$ implies that \eqref{47FS} holds.
Now, take $\psi\in C^{\infty}_{c}(\overline{\R^{N+1}_{+}})$ such that $0\leq \psi\leq 1$, $\psi=1$ in $B_{\frac{1}{2}}^{+}$, $\psi=0$ in $(B_{1}^{+})^{c}$ and $\| \nabla \psi \|_{L^{\infty}(\R^{N+1}_{+})}\leq 2$. 
Fix $i\in I$.
For $\rho>0$, we define $\psi_{\rho}(x, y):=\psi(\frac{x-x_{i}}{\rho}, \frac{y}{\rho})$. Applying \eqref{BRANDLE} to $\psi_{\rho}v_{n}$, we have that
\begin{align*}
S_{*}^{\frac{1}{2}} \left(\int_{\R^{N}} |\psi_{\rho}(x, 0)v_{n}(x, 0)|^{2^{*}_{s}}\, dx\right)^{\frac{1}{\2}}\leq  \left(\iint_{\R^{N+1}_{+}} y^{1-2s}|\nabla (\psi_{\rho}v_{n})|^{2}\, dxdy\right)^{\frac{1}{2}},
\end{align*}
from which
\begin{align*}
&S_{*}^{\frac{1}{2}} \left(\int_{\R^{N}} |\psi_{\rho}(x, 0)v_{n}(x, 0)|^{2^{*}_{s}}\, dx\right)^{\frac{1}{\2}}\\
&\leq  \left(\iint_{\R^{N+1}_{+}} y^{1-2s}|\nabla \psi_{\rho}|^{2}v^{2}_{n}\, dxdy\right)^{\frac{1}{2}}+\left(\iint_{\R^{N+1}_{+}} y^{1-2s}|\nabla v_{n}|^{2}\psi_{\rho}^{2}\, dxdy\right)^{\frac{1}{2}}\\
&\leq  \left(\iint_{\R^{N+1}_{+}} y^{1-2s}|\nabla \psi_{\rho}|^{2}v^{2}_{n}\, dxdy\right)^{\frac{1}{2}}+\left(\iint_{\R^{N+1}_{+}} y^{1-2s}(|\nabla v_{n}|^{2}+m^{2}v_{n}^{2})\psi_{\rho}^{2}\, dxdy\right)^{\frac{1}{2}}.
\end{align*}
Letting $n\ri \infty$ and exploiting the fact that $\psi_{\rho}$ has compact support, we obtain
\begin{align}\label{LG65}
S_{*}^{\frac{1}{2}} \left(\int_{\R^{N}} |\psi_{\rho}(x, 0)|^{2^{*}_{s}}\, d\nu\right)^{\frac{1}{\2}}&\leq  \left(\iint_{\R^{N+1}_{+}} y^{1-2s}|\nabla \psi_{\rho}|^{2}v^{2}\, dxdy\right)^{\frac{1}{2}}+\left(\iint_{\R^{N+1}_{+}} \psi_{\rho}^{2}\, d\mu\right)^{\frac{1}{2}}.
\end{align}
Using the H\"older inequality with exponents $\gamma$ and $\frac{\gamma}{\gamma-1}$ and recalling \eqref{weightedE}, we see that
\begin{align*}
\left(\iint_{\R^{N+1}_{+}} y^{1-2s}|\nabla \psi_{\rho}|^{2}v^{2}\, dxdy\right)^{\frac{1}{2}}&\leq \frac{C}{\rho} \left(\iint_{B^{+}_{\rho}(x_{i}, 0)} y^{1-2s} v^{2} \, dx dy\right)^{\frac{1}{2}} \\
&\leq \frac{C}{\rho} \left(\iint_{B^{+}_{\rho}(x_{i}, 0)} y^{1-2s} |v|^{2\gamma} \, dx dy\right)^{\frac{1}{2\gamma}} \left(\iint_{B^{+}_{\rho}(x_{i}, 0)} y^{1-2s}  \, dx dy\right)^{\frac{\gamma-1}{2\gamma}} \\
&\leq C \left(\iint_{B^{+}_{\rho}(x_{i}, 0)} y^{1-2s} |v|^{2\gamma} \, dx dy\right)^{\frac{1}{2\gamma}}\ri 0 \quad\mbox{ as } \rho\ri 0.
\end{align*}
Then, passing to the limit as $\rho\ri 0$ in \eqref{LG65}, we find
$$
S_{*} \nu_{i}^{\frac{2}{2^{*}_{s}}}\leq \mu_{i}:=\lim_{\rho\ri 0} \mu(B^{+}_{\rho}(x_{i}, 0)),
$$
and so \eqref{49FS} is true.
Since $\mu\geq \sum_{i\in I} \mu_{i}\delta_{(x_{i}, 0)}$, $\mu\geq y^{1-2s} (|\nabla v|^{2}+m^{2}v^{2})$ (by the weak convergence), and $y^{1-2s} (|\nabla v|^{2}+m^{2}v^{2})$ and $\sum_{i\in I} \mu_{i}\delta_{(x_{i}, 0)}$ are orthogonal, we deduce that \eqref{48FS} hold.
\end{proof}

\noindent
Next, we prove the tightness of the Palais-Smale sequences of $J_{\e}$. More precisely, we establish the following result.
\begin{lem}\label{lemTIGHT}
Let $0<c<\frac{s}{N}(\zeta S_{*})^{\frac{N}{2s}}$ and $(v_{n})\subset X_{\e}$ be a Palais-Smale sequence of $J_{\e}$ at the level $c$. Then, for all $\xi>0$, there exists $R=R(\xi)>0$ such that
\begin{align}\label{TIGHT}
\limsup_{n\ri \infty} \left[\iint_{\R^{N+1}_{+}\setminus B^{+}_{R}} y^{1-2s}(|\nabla v_{n}|^{2}+m^{2}v^{2}_{n})\, dx dy+\int_{\R^{N}\setminus B_{R}} (V_{\e}(x)+V_{1})v_{n}^{2}(x, 0)\, dx\right]<\xi.
\end{align}
\end{lem}
\begin{proof}
For $R>0$, let $\eta_{R}\in C^{\infty}(\overline{\R^{N+1}_{+}})$ be a function such that
\begin{equation*}
\eta_{R}(x, y):=
\left\{
\begin{array}{ll}
0 &\mbox{ if } (x, y)\in B^{+}_{R/2}, \\
1 &\mbox{ if } (x, y)\in \overline{\R^{N+1}_{+}}\setminus B^{+}_{R},
\end{array}
\right.
\end{equation*}
with $0\leq \eta_{R}\leq 1$ and $\|\nabla \eta_{R}\|_{L^{\infty}(\R^{N+1}_{+})}\leq C/R$, for some $C>0$ independent of $R>0$. Since $(v_{n})$ is a bounded Palais-Smale sequence in $X_{\e}$, we see that $\langle J'_{\e}(v_{n}), v_{n}\eta^{2}_{R}\rangle=o_{n}(1)$, i.e., 
\begin{align*}
&\iint_{\R^{N+1}_{+}} y^{1-2s} (|\nabla v_{n}|^{2}+m^{2}v_{n}^{2})\eta^{2}_{R}\, dx dy+\int_{\R^{N}} V_{\e}(x)v^{2}_{n}(x, 0)\eta^{2}_{R}(x, 0)\, dx\\
&=-2\iint_{\R^{N+1}_{+}} y^{1-2s} v_{n}\eta_{R}\nabla v_{n}\cdot \nabla \eta_{R}\, dx dy+\int_{\R^{N}} g_{\e}(x, v_{n}(x, 0)) v_{n}(x, 0)\eta^{2}_{R}(x, 0)\, dx+o_{n}(1),
\end{align*}
which can be rewritten as:
\begin{align}\label{HB1}
&\iint_{\R^{N+1}_{+}} y^{1-2s} (|\nabla v_{n}|^{2}+m^{2}v_{n}^{2})\eta^{2}_{R}\, dx dy+\int_{\R^{N}} (V_{\e}(x)+V_{1})v^{2}_{n}(x, 0)\eta^{2}_{R}(x, 0)\, dx \nonumber\\
&=\int_{\R^{N}} g_{\e}(x, v_{n}(x, 0)) v_{n}(x, 0)\eta^{2}_{R}(x, 0)\, dx+V_{1} \int_{\R^{N}} v^{2}_{n}(x, 0)\eta^{2}_{R}(x, 0)\, dx \nonumber \\
&\quad-2\iint_{\R^{N+1}_{+}} y^{1-2s} v_{n}\eta_{R}\nabla v_{n}\cdot \nabla \eta_{R}\, dx dy+o_{n}(1).
\end{align}
Choose $R>0$ such that $\Lambda_{\e}\subset B_{R/2}$. Thus, thanks to $(g_{3})$-$(ii)$, 
\begin{align}\label{HB2}
\int_{\R^{N}} g_{\e}(x, v_{n}(x, 0)) v_{n}(x, 0)\eta^{2}_{R}(x, 0)\, dx+V_{1}  \int_{\R^{N}} v^{2}_{n}(x, 0)\eta^{2}_{R}(x, 0)\, dx\leq V_{1} \left(1+\frac{1}{\kappa} \right)\int_{\R^{N}} v_{n}^{2}(x, 0)\eta^{2}_{R}(x, 0)\, dx.
\end{align}
On the other hand, by H\"older's inequality, $0\leq \eta_{R}\leq 1$,  $\|\nabla \eta_{R}\|_{L^{\infty}(\R^{N+1}_{+})}\leq C/R$, and the boundedness of $(v_{n})$ in $X_{\e}$, we have
\begin{align}\label{HB3}
\left|2\iint_{\R^{N+1}_{+}} y^{1-2s} v_{n} \eta_{R}\nabla v_{n}\cdot \nabla \eta_{R}\, dx dy\right|&\leq \frac{C}{R} \iint_{\R^{N+1}_{+}} y^{1-2s}v_{n}|\nabla v_{n}|\, dx dy \nonumber\\
&\leq \frac{C}{R} \left(\iint_{\R^{N+1}_{+}} y^{1-2s}v^{2}_{n}\, dx dy  \right)^{\frac{1}{2}}  \left(\iint_{\R^{N+1}_{+}} y^{1-2s}|\nabla v_{n}|^{2}\, dx dy  \right)^{\frac{1}{2}} \nonumber\\
&\leq \frac{C}{R}.
\end{align}
Then, applying \eqref{Young} to $v_{n}\eta_{R}$, and using $\|\nabla \eta_{R}\|_{L^{\infty}(\R^{N+1}_{+})}\leq C/R$ and \eqref{HB3}, we infer that
\begin{align}\label{HB4}
&V_{1} \left(1+\frac{1}{\kappa} \right)\int_{\R^{N}} v^{2}_{n}(x, 0)\eta^{2}_{R}(x, 0)\, dx \nonumber \\ 
&\leq V_{1} m^{-2s}\left(1+\frac{1}{\kappa} \right) \iint_{\R^{N+1}_{+}} y^{1-2s} |\nabla (v_{n}\eta_{R})|^{2}\, dx dy+V_{1}m^{2-2s}\left(1+\frac{1}{\kappa} \right) \iint_{\R^{N+1}_{+}} y^{1-2s} v^{2}_{n}\eta^{2}_{R}\, dx dy \nonumber \\
&= V_{1} m^{-2s}\left(1+\frac{1}{\kappa} \right) \iint_{\R^{N+1}_{+}} y^{1-2s} |\nabla v_{n}|^{2}\eta^{2}_{R}\, dx dy+V_{1} m^{-2s}\left(1+\frac{1}{\kappa} \right) \iint_{\R^{N+1}_{+}} y^{1-2s} |\nabla \eta_{R}|^{2}v^{2}_{n}\, dx dy \nonumber \\
&\quad+2V_{1} m^{-2s} \left(1+\frac{1}{\kappa} \right) \iint_{\R^{N+1}_{+}} y^{1-2s}  v_{n}\eta_{R} \nabla v_{n}\cdot \nabla\eta_{R}\, dx dy+V_{1}m^{2-2s} \left(1+\frac{1}{\kappa} \right)\iint_{\R^{N+1}_{+}} y^{1-2s} v^{2}_{n}\eta^{2}_{R}\, dx dy \nonumber \\
&\leq V_{1} m^{-2s} \left(1+\frac{1}{\kappa} \right) \iint_{\R^{N+1}_{+}} y^{1-2s} (|\nabla v_{n}|^{2}+m^{2}v_{n}^{2})\eta^{2}_{R}\, dx dy+\frac{C}{R^{2}} \iint_{\R^{N+1}_{+}} y^{1-2s}v_{n}^{2}\, dxdy+\frac{C}{R} \nonumber \\
&\leq V_{1} m^{-2s}\left(1+\frac{1}{\kappa} \right) \iint_{\R^{N+1}_{+}} y^{1-2s} (|\nabla v_{n}|^{2}+m^{2}v_{n}^{2})\eta^{2}_{R}\, dx dy+\frac{C}{R^{2}}+\frac{C}{R}.
\end{align}
Putting together \eqref{HB1}-\eqref{HB4}, we arrive at
\begin{align}\label{HB5}
&\left[1-V_{1} m^{-2s}\left(1+\frac{1}{\kappa} \right)\right] \iint_{\R^{N+1}_{+}} y^{1-2s} (|\nabla v_{n}|^{2}+m^{2}v_{n}^{2})\eta^{2}_{R}\, dx dy+\int_{\R^{N}} (V_{\e}(x)+V_{1})v^{2}_{n}(x, 0)\eta^{2}_{R}(x, 0)\, dx \nonumber\\
&\leq \frac{C}{R}+\frac{C}{R^{2}}+o_{n}(1).
\end{align}
By means of $\kappa>\frac{V_{1}}{m^{2s}-V_{1}}$, $(V_{1})$ and the definition of $\eta_{R}$, we deduce that \eqref{HB5} implies the assertion.
\end{proof}

\begin{remark}
Differently from \cite{AF, AmbrosioBOOK, DF}, we use $\langle J'_{\e}(v_{n}), v_{n}\eta^{2}_{R}\rangle=o_{n}(1)$ instead of $\langle J'_{\e}(v_{n}), v_{n}\eta_{R}\rangle=o_{n}(1)$ to obtain \eqref{TIGHT}. This is motivated by the fact that to estimate the quadratic terms, it is needed to apply in a careful way Inequality \eqref{Young}.
\end{remark}

\begin{remark}\label{ADELAIDE}
Let $r\in[2, \2]$. Exploiting $\eta_{R}(\cdot, 0)=1$ in $B_{R}^{c}$, Theorem \ref{Sembedding}, $|x+y|^{2}\leq 2(|x|^{2}+|y|^{2})$ for all $x, y\in \R^{N}$, $\|\nabla \eta_{R}\|_{L^{\infty}(\R^{N+1}_{+})}\leq C/R$ and the boundedness of $(v_{n})$ in $X_{\e}$, we can see that
\begin{align*}
&\left(\int_{B^{c}_{R}} |v_{n}(x, 0)|^{r}\, dx  \right)^{\frac{2}{r}} \leq \left(\int_{\R^{N}} |(v_{n}\eta_{R})(x, 0)|^{r}\, dx  \right)^{\frac{2}{r}} \\
&\leq C \left(\iint_{\R^{N+1}_{+}} y^{1-2s}(|\nabla(v_{n}\eta_{R})|^{2}+m^{2}v^{2}_{n}\eta_{R}^{2})\, dxdy\right)\\
&\leq C\left(\iint_{\R^{N+1}_{+}} y^{1-2s}(|\nabla v_{n}|^{2}+m^{2}v_{n}^{2})\eta^{2}_{R}\, dxdy+\iint_{\R^{N+1}_{+}} y^{1-2s}|\nabla \eta_{R}|^{2}v^{2}_{n}\, dxdy\right) \\
&\leq C\left(\iint_{\R^{N+1}_{+}} y^{1-2s}(|\nabla v_{n}|^{2}+m^{2}v_{n}^{2})\eta^{2}_{R}\, dxdy+\frac{C}{R^{2}}\right),
\end{align*}
which combined with \eqref{HB5} gives
\begin{align*}
\left(\int_{B^{c}_{R}} |v_{n}(x, 0)|^{r}\, dx  \right)^{\frac{2}{r}}\leq \frac{C}{R}+\frac{C}{R^{2}}+o_{n}(1).
\end{align*}
Thus,
\begin{align}\label{HB6}
\lim_{R\ri \infty} \limsup_{n\ri \infty} \int_{B^{c}_{R}} |v_{n}(x, 0)|^{r}\, dx=0, \quad \mbox{ for all } r\in [2, \2].
\end{align}
\end{remark}

\noindent
At this point, we can show that the modified functional fulfills a local compactness condition.
\begin{lem}\label{lemma2}
Let $0<c<\frac{s}{N}(\zeta S_{*})^{\frac{N}{2s}}$. Then, $J_{\e}$ satisfies the Palais-Smale condition at the level $c$.
\end{lem}
\begin{proof}
Let $0<c<\frac{s}{N}(\zeta S_{*})^{\frac{N}{2s}}$ and $(v_{n})\subset X_{\e}$ be a Palais-Smale sequence at the level $c$, namely,
\begin{align*}
J_{\e}(v_{n})\ri c \quad \mbox{ and } \quad J'_{\e}(v_{n})\ri 0 \mbox{ in } X^{*}_{\e},
\end{align*}
as  $n\ri \infty$. By Lemma \ref{boundPS}, we know that $(v_{n})$ is bounded in $X_{\e}$. 
In view of Remark \ref{REMKINDERF}, we can suppose that $(v_{n})$ is nonnegative.
Thanks to the reflexivity of $X_{\e}$ and Theorem \ref{Sembedding}, up to a subsequence, we may assume that  
\begin{equation}\label{OSVALDO}
\left\{
\begin{array}{ll}
v_{n}\rightharpoonup v &\mbox{ in } X_{\e}, \\
v_{n}(\cdot, 0)\ri v(\cdot, 0) &\mbox{ in } L^{r}_{loc}(\R^{N}), \,\, \mbox{ for all } r\in [1, \2), \\
v_{n}(\cdot, 0)\ri v(\cdot, 0) &\mbox{ a.e. in } \R^{N}.
\end{array}
\right.
\end{equation}
We are going to demonstrate that $v_{n}\ri v$ in $X_{\e}$. 
Using the density of $C^{\infty}_{c}(\overline{\R^{N+1}_{+}})$ in $X_{\e}$, $(g_{1})$, $(g_{2})$, $(f_{2})$, and \eqref{OSVALDO}, it is easy to check that $\langle J'_{\e}(v), \varphi\rangle=0$ for all $\varphi\in X_{\e}$. In particular,
\begin{align}\label{AM1}
\|v\|^{2}_{\e}=\int_{\R^{N}} g_{\e}(x, v(x, 0)) v(x, 0)\, dx.
\end{align}
On the other hand, $\langle J'_{\e}(v_{n}), v_{n}\rangle=o_{n}(1)$, i.e., 
\begin{align}\label{AM2}
\|v_{n}\|^{2}_{\e}=\int_{\R^{N}} g_{\e}(x, v_{n}(x, 0)) v_{n}(x, 0)\, dx+o_{n}(1).
\end{align}
In light of \eqref{AM1} and \eqref{AM2}, if we prove that
\begin{align}\label{AM3}
\int_{\R^{N}} g_{\e}(x, v_{n}(x, 0)) v_{n}(x, 0)\, dx=\int_{\R^{N}} g_{\e}(x, v(x, 0)) v(x, 0)\, dx+o_{n}(1),
\end{align}
then we deduce that $\|v_{n}\|_{\e}\ri \|v\|_{\e}$ as $n\ri \infty$, and recalling that $X_{\e}$ is a Hilbert space, we conclude that $v_{n}\ri v$ in $X_{\e}$ as $n\ri \infty$. Next, we verify that \eqref{AM3} is valid. By virtue of Lemma \ref{lemTIGHT}, fixed $\xi>0$, there exists $R=R(\xi)>0$ such that \eqref{TIGHT} is true. By $(g_{2})$, $(f_{1})$, $(f_{2})$, and \eqref{HB6}, we see that 
\begin{align}\label{HB9}
\limsup_{n\ri \infty} \left|\int_{B^{c}_{R}} g_{\e}(x, v_{n}(x, 0))v_{n}(x, 0)\, dx \right| 
&\leq C\limsup_{n\ri \infty} \int_{B^{c}_{R}} ( |v_{n}(x, 0)|^{2}+|v_{n}(x, 0)|^{q}+|v_{n}(x, 0)|^{\2})\, dx \nonumber\\
&\leq C\xi.
\end{align}
On the other hand, because $g_{\e}(\cdot, v(\cdot, 0))v(\cdot, 0)\in L^{1}(\R^{N})$, we can take $R>0$ large enough so that
\begin{align}\label{HB10}
\int_{B^{c}_{R}} g_{\e}(x, v(x, 0)) v(x, 0)\, dx\leq \xi.
\end{align}
Thus, \eqref{HB9} and \eqref{HB10} yield
\begin{align}\label{HB11}
\limsup_{n\ri \infty} \left|\int_{B^{c}_{R}} g_{\e}(x, v_{n}(x, 0))v_{n}(x, 0)\, dx-\int_{B^{c}_{R}} g_{\e}(x, v(x, 0))v(x, 0)\, dx \right|\leq C\xi.
\end{align}
Now, it follows from  the definition of $g$ that
\begin{align*}
g_{\e}(x, v_{n}(x, 0))v_{n}(x, 0)\leq f(v_{n}(x, 0))v_{n}(x, 0)+a^{\2}+\frac{V_{1}}{\kappa}v_{n}^{2}(x, 0), \quad \mbox{ for a.e. } x\in \Lambda^{c}_{\e}.
\end{align*}
Since $B_{R}\cap \Lambda^{c}_{\e}$ is bounded, we can use the aforementioned estimate, $(f_{1})$, $(f_{2})$, \eqref{OSVALDO}, and the dominated convergence theorem to infer that, as $n\ri \infty$,
\begin{align}\label{HB12}
\int_{B_{R}\cap \Lambda^{c}_{\e}} g_{\e}(x, v_{n}(x, 0)) v_{n}(x, 0)\, dx=\int_{B_{R}\cap \Lambda^{c}_{\e}} g_{\e}(x, v(x, 0)) v(x, 0)\, dx+o_{n}(1).
\end{align}
At this point, we aim to show that
\begin{align}\label{HB13}
\int_{\Lambda_{\e}} v_{n}^{\2}(x, 0)\, dx=\int_{\Lambda_{\e}} v^{\2}(x, 0)\, dx+o_{n}(1).
\end{align}
In fact, if we assume that \eqref{HB13} holds, then we can exploit $(g_{2})$, $(f_{1})$, $(f_{2})$, \eqref{OSVALDO}, and the dominated convergence theorem again to obtain 
\begin{align*}
\int_{B_{R}\cap \Lambda_{\e}} g_{\e}(x, v_{n}(x, 0)) v_{n}(x, 0)\, dx=\int_{B_{R}\cap \Lambda_{\e}} g_{\e}(x, v(x, 0)) v(x, 0)\, dx+o_{n}(1),
\end{align*}
which combined with \eqref{HB11} and \eqref{HB12} gives \eqref{AM3}. Therefore, we shall prove that \eqref{HB13} is satisfied. Taking into account that $(v_{n})$ is a bounded tight sequence in $\x$, we may suppose that
\begin{align}\begin{split}\label{HB15}
&y^{1-2s} (|\nabla v_{n}|^{2}+m^{2}v_{n}^{2})\rightharpoonup \mu \quad \mbox{ weakly in the sense of measures, } \\
&v^{\2}_{n}(\cdot, 0)\rightharpoonup \nu \quad \mbox{ weakly in the sense of measures, }
\end{split}\end{align}
where $\mu$ and $\nu$ are two bounded nonnegative measures on $\R^{N+1}_{+}$ and $\R^{N}$, respectively. 
Thus, applying Proposition \ref{CCL}, we can find an at most countable index set $I$ and sequences $(x_{i})_{i\in I}\subset \R^{N}$, $(\mu_{i})_{i\in I}\subset (0, \infty)$, and $(\nu_{i})_{i\in I} \subset (0, \infty)$ such that
\begin{align}\begin{split}\label{HB16}
&\nu=v^{\2}(\cdot, 0)+\sum_{i\in I} \nu_{i} \delta_{x_{i}}, \\
&\mu\geq y^{1-2s} (|\nabla v|^{2}+m^{2}v^{2})+\sum_{i\in I} \mu_{i} \delta_{(x_{i}, 0)},  \\
&\mu_{i}\geq S_{*} \nu_{i}^{\frac{2}{2^{*}_{s}}}, \quad \mbox{ for all }  i\in I. 
\end{split}\end{align}
Let us show that $(x_{i})_{i\in I}\cap \Lambda_{\e}=\emptyset$. Assume, by contradiction, that $x_{i}\in \Lambda_{\e}$ for some $i\in I$. Fixed $\rho>0$, we define $\psi_{\rho}(x, y):=\psi(\frac{x-x_{i}}{\rho}, \frac{y}{\rho})$, where $\psi\in C^{\infty}_{c}(\overline{\R^{N+1}_{+}})$ is such that $\psi=1$ in $B_{\frac{1}{2}}^{+}$ and $\psi=0$ in $(B_{1}^{+})^{c}$, $0\leq \psi\leq 1$, and $\| \nabla \psi \|_{L^{\infty}(\R^{N+1}_{+})}\leq 2$. We suppose that $\rho>0$ is such that ${\rm supp}(\psi_{\rho}(\cdot, 0))\subset \Lambda_{\e}$. Because $(\psi_{\rho} v_{n})$ is bounded in $X_{\e}$, we have that $\langle J'_{\e}(v_{n}), \psi_{\rho}v_{n}\rangle=o_{n}(1)$, and so
\begin{align*}
&\iint_{\R^{N+1}_{+}} y^{1-2s}(|\nabla v_{n}|^{2}+m^{2}v_{n}^{2})\psi_{\rho}\, dx dy+\int_{\R^{N}} V_{\e}(x)v^{2}_{n}(x, 0)\psi_{\rho}(x, 0)\, dx \nonumber\\
&=-\iint_{\R^{N+1}_{+}} y^{1-2s} v_{n}\nabla v_{n}\cdot \nabla \psi_{\rho}\, dx dy+\int_{\R^{N}} f(v_{n}(x, 0)) \psi_{\rho}(x, 0) v_{n}(x, 0)\, dx+\int_{\R^{N}} \psi_{\rho}(x, 0) v^{\2}_{n}(x, 0) \, dx+o_{n}(1),
\end{align*}
or equivalently,
\begin{align}\label{HB17}
&\iint_{\R^{N+1}_{+}} y^{1-2s}(|\nabla v_{n}|^{2}+m^{2}v_{n}^{2})\psi_{\rho}\, dx dy+\int_{\R^{N}} (V_{\e}(x)+V_{1})v^{2}_{n}(x, 0)\psi_{\rho}(x, 0)\, dx \nonumber\\
&=V_{1}\int_{\R^{N}} v^{2}_{n}(x, 0) \psi_{\rho}(x, 0)\, dx+\int_{\R^{N}} f(v_{n}(x, 0)) \psi_{\rho}(x, 0) v_{n}(x, 0)\, dx+\int_{\R^{N}} \psi_{\rho}(x, 0) v^{\2}_{n}(x, 0) \, dx \nonumber\\
&\quad -\iint_{\R^{N+1}_{+}} y^{1-2s} v_{n} \nabla v_{n}\cdot \nabla \psi_{\rho}\, dx dy+o_{n}(1).
\end{align}
Since $f$ has subcritical growth and $\psi_{\rho}(\cdot, 0)$ has compact support, we can use \eqref{OSVALDO} to see that
\begin{align}\label{HB18}
\lim_{\rho \ri 0} \lim_{n\ri \infty} \int_{\R^{N}} f(v_{n}(x, 0)) \psi_{\rho}(x, 0) v_{n}(x, 0)\, dx=\lim_{\rho \ri 0} \int_{\R^{N}} f(v(x, 0)) \psi_{\rho}(x, 0) v(x, 0)\, dx=0
\end{align}
and 
\begin{align}\label{HB18+}
\lim_{\rho \ri 0} \lim_{n\ri \infty} \int_{\R^{N}} v^{2}_{n}(x, 0)\psi_{\rho}(x, 0)\, dx=\lim_{\rho \ri 0} \int_{\R^{N}} v^{2}(x, 0) \psi_{\rho}(x, 0)\, dx=0.
\end{align}
Now, we prove that
\begin{align}\label{HB19}
\lim_{\rho \ri 0} \limsup_{n\ri \infty} \iint_{\R^{N+1}_{+}} y^{1-2s} v_{n}\nabla v_{n}\cdot \nabla \psi_{\rho}\, dx dy=0.
\end{align}
From the H\"older inequality, $(v_{n})$ is bounded in $\x$, $\x$ is compactly embedded in $L^{2}(B_{\rho}^{+}(x_{i}, 0), y^{1-2s})$, and $\|\nabla \psi_{\rho}\|_{L^{\infty}(\R^{N+1}_{+})}\leq \frac{C}{\rho}$, we obtain that
\begin{align*}
&\limsup_{n\ri \infty}\left| \iint_{\R^{N+1}_{+}} y^{1-2s} v_{n}\nabla v_{n}\cdot \nabla \psi_{\rho}\, dx dy \right| \\
&\leq \limsup_{n\ri \infty} \left( \iint_{\R^{N+1}_{+}} y^{1-2s} |\nabla v_{n}|^{2}\, dx dy\right)^{\frac{1}{2}} \left( \iint_{B^{+}_{\rho}(x_{i}, 0)} y^{1-2s} v^{2}_{n} |\nabla \psi_{\rho}|^{2}\, dx dy\right)^{\frac{1}{2}}  \\
&\leq \frac{C}{\rho} \left( \iint_{B^{+}_{\rho}(x_{i}, 0)} y^{1-2s} v^{2} \, dx dy\right)^{\frac{1}{2}}.
\end{align*}
Applying the H\"older inequality with exponents $\gamma$ and $\frac{\gamma}{\gamma-1}$ and bearing in mind \eqref{weightedE}, we have that
\begin{align*}
&\frac{C}{\rho} \left(\iint_{B^{+}_{\rho}(x_{i}, 0)} y^{1-2s} v^{2} \, dx dy\right)^{\frac{1}{2}} \\
&\leq \frac{C}{\rho} \left(\iint_{B^{+}_{\rho}(x_{i}, 0)} y^{1-2s} v^{2\gamma} \, dx dy\right)^{\frac{1}{2\gamma}} \left(\iint_{B^{+}_{\rho}(x_{i}, 0)} y^{1-2s}  \, dx dy\right)^{\frac{\gamma-1}{2\gamma}} \\
&\leq C \left(\iint_{B^{+}_{\rho}(x_{i}, 0)} y^{1-2s} v^{2\gamma} \, dx dy\right)^{\frac{1}{2\gamma}}\ri 0 \quad\mbox{ as } \rho\ri 0.
\end{align*}
The aforementioned estimates show that \eqref{HB19} is satisfied. Therefore, from \eqref{HB16} and taking the limit as $n\ri \infty$ and $\rho\ri 0$ in \eqref{HB17}, we deduce that \eqref{HB18}, \eqref{HB18+}, \eqref{HB19} and $(V_{1})$ yield $\mu_{i}\leq \nu_{i}$. This fact combined with the last statement in \eqref{HB16} implies that
\begin{align}\label{HB19+}
\nu_{i}\geq S_{*}^{\frac{N}{2s}}.
\end{align}
Hence, exploiting $(f_{4})$, $(g_{3})$, ${\rm supp}(\psi_{\rho}(\cdot, 0))\subset \Lambda_{\e}$, and $0\leq \psi_{\rho}\leq 1$, we arrive at
\begin{align*}
c&=J_{\e}(v_{n})-\frac{1}{2}\langle J'_{\e}(v_{n}), v_{n}\rangle +o_{n}(1) \\
&=\int_{\Lambda^{c}_{\e}} \left[\frac{1}{2}g_{\e}(x, v_{n}(x, 0))v_{n}(x, 0)-G_{\e}(x, v_{n}(x, 0)) \right]\, dx+\int_{\Lambda_{\e}} \left[\frac{1}{2}f(v_{n}(x, 0))v_{n}(x, 0)-F(v_{n}(x, 0)) \right]\, dx\\
&\quad +\frac{s}{N} \int_{\Lambda_{\e}} v_{n}^{\2}(x, 0) \, dx+o_{n}(1) \\
&\geq  \frac{s}{N} \int_{\Lambda_{\e}} v_{n}^{\2}(x, 0) \, dx+o_{n}(1) \\
&\geq \frac{s}{N} \int_{\Lambda_{\e}} \psi_{\rho}(x, 0)v_{n}^{\2}(x, 0) \, dx+o_{n}(1).
\end{align*} 
Letting $n\ri \infty$ and using \eqref{HB16} and \eqref{HB19+}, we obtain
\begin{align*}
c\geq \frac{s}{N} \sum_{\{i\in I: x_{i}\in \Lambda_{\e} \}} \psi_{\rho}(x_{i}, 0) \nu_{i}\geq \frac{s}{N} \nu_{i}\geq \frac{s}{N} S_{*}^{\frac{N}{2s}},
\end{align*}
which leads to a contradiction because $c<\frac{s}{N}(\zeta S_{*})^{\frac{N}{2s}}$ and $\zeta\in (0, 1)$. Consequently, \eqref{HB13} holds, and we can conclude that $v_{n}\ri v$ in $X_{\e}$.
\end{proof}

\noindent
In light of Lemmas \ref{lemma1} and  \ref{lemma2}, we can apply the mountain pass theorem \cite{AR} 
to infer the next existence result for \eqref{MEP}.
\begin{thm}\label{lemAM2.8}
For all $\e>0$, there exists a nonnegative function $v_{\e}\in X_{\e}\setminus\{0\}$ such that 
\begin{align}\label{AM2.8}
J_{\e}(v_{\e})=c_{\e} \quad \mbox{ and } \quad J'_{\e}(v_{\e})=0.
\end{align} 
\end{thm}

\begin{remark}
It is possible to give an alternative proof of Theorem \ref{lemAM2.8} without using Proposition \ref{CCL}. 
We outline the details. Let $(v_{n})\subset X_{\e}$ be a Palais-Smale sequence at the mountain pass level $c_{\e}$. By Lemma \ref{boundPS}, we know that $(v_{n})$ is bounded in $X_{\e}$. Exploiting $J_{\e}(v_{n})\ri c_{\e}>0$ and $\langle J'_{\e}(v_{n}), v_{n}\rangle=o_{n}(1)$, we can argue as in the proof of Lemma \ref{lem2.4AM} to prove that there exist $(z_{n})\subset \R^{N}$ and $r, \beta>0$ such that 
\begin{align*}
\liminf_{n\ri \infty}\int_{B_{r}(z_{n})} v_{n}^{2}(x, 0)\, dx\geq \beta.
\end{align*}
On the other hand, reasoning as in Remark \ref{ADELAIDE}, we have that for all $R>0$ such that $\Lambda_{\e}\subset B_{R/2}$, 
$$
\int_{B^{c}_{R}} v_{n}^{2}(x, 0)\, dx\leq \frac{C}{R}+\frac{C}{R^{2}}+o_{n}(1).
$$
In view of the aforementioned estimates, we deduce that the sequence $(z_{n})$ is bounded in $\R^{N}$. 
Now, because $(v_{n})$ is bounded in $X_{\e}$, up to a subsequence, we may assume that there exists $v_{\e}\in X_{\e}$ such that $v_{n}\rightharpoonup v_{\e}$ in $X_{\e}$, $v_{n}(\cdot, 0)\ri v_{\e}(\cdot, 0)$ in $L^{r}_{loc}(\R^{N})$ for all $r\in [1, \2)$, and $v_{n}(\cdot, 0)\ri v_{\e}(\cdot, 0)$ a.e. in $\R^{N}$. Then, it is easy to check that $v_{\e}$ is a critical point of $J_{\e}$. Moreover, we can see that $v_{\e}\not\equiv 0$. In fact, due to the boundedness of $(z_{n})$, we can find $k>0$ such that $B_{r}(z_{n})\subset B_{k}$ for all $n\in \mathbb{N}$. Consequently, 
$$
\int_{B_{k}} v_{\e}^{2}(x, 0)\, dx=\liminf_{n\ri \infty}\int_{B_{k}} v_{n}^{2}(x, 0)\, dx\geq \liminf_{n\ri \infty}\int_{B_{r}(z_{n})} v_{n}^{2}(x, 0)\, dx\geq \beta,
$$ 
which implies that $v_{\e}\not\equiv 0$.  Finally, to verify that $J_{\e}(v_{\e})=c_{\e}$, it suffices to use $(g_3)$ and Fatou's lemma to get 
\begin{align*}
c_{\e}\leq J_{\e}(v_{\e})-\frac{1}{2}\langle J'_{\e}(v_{\e}), v_{\e}\rangle\leq \liminf_{n\ri \infty} \left[J_{\e}(v_{n})-\frac{1}{2}\langle J'_{\e}(v_{n}), v_{n}\rangle\right]=c_{\e}.
\end{align*}
\end{remark}

\section{Autonomous critical problems}
Let $\mu>-V_{1}(1+\frac{1}{\kappa})$ and introduce the following autonomous problem related to \eqref{P}:
\begin{align}\label{AIKCP}
\left\{
\begin{array}{ll}
(-\Delta+m^{2})^{s}u+\mu u=f(u)+u^{\2-1} \mbox{ in } \R^{N}, \\
u\in H^{s}(\R^{N}), \quad u>0 \mbox{ in } \R^{N}.
\end{array}
\right.
\end{align}
The extended problem associated with \eqref{AIKCP} is given by:
\begin{align}\label{AEP}
\left\{
\begin{array}{ll}
-\dive(y^{1-2s} \nabla v)+m^{2}y^{1-2s}v=0 &\mbox{ in } \R^{N+1}_{+}, \\
\frac{\partial v}{\partial \nu^{1-2s}}= -\mu v(\cdot, 0)+f(v(\cdot, 0))+(v^{+}(\cdot, 0))^{\2-1} &\mbox{ on } \R^{N},
\end{array}
\right.
\end{align}
and the corresponding energy functional
\begin{align*}
L_{\mu}(v):=\frac{1}{2} \|v\|^{2}_{Y_{\mu}}-\int_{\R^{N}} \left[F(v(x, 0))+\frac{1}{\2}(v^{+}(x, 0))^{\2}\right]\, dx
\end{align*}
 is well defined on $Y_{\mu}:=\x$ endowed with the norm: 
\begin{align*}
\|v\|_{Y_{\mu}}:=\left( \|v\|^{2}_{\x}+\mu |v(\cdot, 0)|^{2}_{2}  \right)^{\frac{1}{2}}.
\end{align*}
To check that $\|\cdot\|_{Y_{\mu}}$ is a norm equivalent to $\|\cdot\|_{\x}$, one can argue as at pag. 5671 in \cite{ADCDS} (observe that $\mu>-V_{1}(1+\frac{1}{\kappa})>-m^{2s}$). We also note that, by \eqref{m-ineq}, for all $v\in Y_{\mu}$,
\begin{align}\label{DUPRESS}
\|v\|_{Y_{\mu}}^{2}\geq \zeta \|v\|_{\x}^{2}. 
\end{align}
Obviously, $Y_{\mu}$ is a Hilbert space with the inner product
\begin{align*}
\langle v, w\rangle_{Y_{\mu}}:=\iint_{\R^{N+1}_{+}} y^{1-2s} (\nabla v\cdot\nabla w+m^{2}vw)\, dx dy+\mu\int_{\R^{N}} v(x, 0) w(x, 0)\, dx, \quad\mbox{ for all } v, w\in Y_{\mu}.
\end{align*}
Denote by $\mathcal{M}_{\mu}$ the Nehari manifold associated with $L_{\mu}$, i.e.,
$$
\mathcal{M}_{\mu}:=\{v\in Y_{\mu}: \langle L'_{\mu}(v), v\rangle=0\}.
$$
As in the previous section, it is easy to verify that $L_{\mu}$ has a mountain pass geometry \cite{AR}. 
Thus, invoking a variant of the mountain pass theorem without the Palais-Smale condition (see \cite[Theorem 2.9]{W}),
we can find a Palais-Smale sequence $(v_{n})\subset \y$ at the mountain pass level $d_{\mu}$ of $L_{\mu}$ given by:
\begin{align*}
d_{\mu}:=\inf_{\gamma\in \Gamma_{\mu}} \max_{t\in [0, 1]} L_{\mu}(\gamma(t)),
\end{align*}
where 
$$
\Gamma_{\mu}:=\{\gamma\in C([0, 1], Y_{\mu}): \gamma(0)=0, \, L_{\mu}(\gamma(1))<0\}.
$$
We stress that $(v_{n})$ is bounded in $Y_{\mu}$. In fact, by $(f_3)$, 
\begin{align*}
C(1+\|v_{n}\|_{\y})&\geq L_{\mu}(v_{n})-\frac{1}{\theta}\langle L'_{\mu}(v_{n}), v_{n}\rangle\\
&=\left(\frac{1}{2}-\frac{1}{\theta}\right) \|v_{n}\|^{2}_{\y}+\frac{1}{\theta} \int_{\R^{N}} [f(v_{n}(x, 0))v_{n}(x, 0)-\theta F(v_{n}(x, 0))]\, dx+\left(\frac{1}{\theta}-\frac{1}{\2} \right)|v^{+}_{n}(\cdot, 0)|^{\2}_{\2}\\
&\geq \left(\frac{1}{2}-\frac{1}{\theta}\right) \|v_{n}\|^{2}_{\y}, 
\end{align*}
which implies the boundedness of $(v_{n})$ in $\y$.  
As in \cite{Rab, W}, by assumptions on $f$, we can see that
$$
0<d_{\mu}=\inf_{v\in \mathcal{M}_{\mu}} L_{\mu}(v)=\inf_{v\in \y\setminus\{0\}} \max_{t\geq 0}L_{\mu}(t v).
$$ 
Arguing as in the proof of Lemma \ref{UPPERB}, it is easy to prove that 
\begin{align}\label{BREZISN}
0<d_{\mu}<\frac{s}{N}(\zeta S_{*})^{\frac{N}{2s}}.
\end{align}
Our claim is to establish the existence of a ground state solution for \eqref{AEP}.
We start by recalling a vanishing Lions-type result.
\begin{lem}\label{Lions}\cite[Lemma 3.3]{ADCDS}
Let $t\in [2, \2)$ and $R>0$. If $(v_{n})\subset \x$ is a bounded sequence such that
$$
\lim_{n\ri \infty} \sup_{z\in \R^{N}} \int_{B_{R}(z)} |v_{n}(x, 0)|^{t}\, dx=0,
$$
then $v_{n}(\cdot, 0)\ri 0$ in $L^{r}(\R^{N})$ for all $r\in (2, \2)$.
\end{lem}

\noindent
The next lemma is a critical version of \cite[Lemma 3.4]{ADCDS}.
\begin{lem}\label{Lions2}
Let $(v_{n})\subset \y$ be a Palais-Smale sequence for $L_{\mu}$ at the level $c<\frac{s}{N}(\zeta S_{*})^{\frac{N}{2s}}$ and such that $v_{n}\rightharpoonup 0$ in $\y$. Then, we have either
\begin{compactenum}[$(a)$]
\item $v_{n}\ri 0$ in $\y$, or
\item there exist a sequence $(z_{n})\subset \R^{N}$ and constants $R, \beta>0$ such that 
$$
\liminf_{n\ri \infty} \int_{B_{R}(z_{n})} v_{n}^{2}(x, 0)\, dx\geq \beta.
$$
\end{compactenum}
\end{lem}
\begin{proof}
Assume that $(b)$ does not occur. Therefore, for all $R>0$, it holds
$$
\lim_{n\ri \infty} \sup_{z\in \R^{N}} \int_{B_{R}(z)} v_{n}^{2}(x, 0)\, dx=0.
$$
By Lemma \ref{Lions}, we know that $v_{n}(\cdot, 0)\ri 0$ in $L^{r}(\R^{N})$ for all $r\in (2, \2)$. 
This fact and $(f_1)$-$(f_2)$ imply that 
\begin{align}\label{HB20}
\int_{\R^{N}}f(v_{n}(x, 0))v_{n}(x, 0)\, dx=\int_{\R^{N}}F(v_{n}(x, 0))\, dx=o_{n}(1).
\end{align}
Exploiting $\langle L'_{\mu}(v_{n}), v_{n}\rangle=o_{n}(1)$ and \eqref{HB20}, we see that 
\begin{align*}
\|v_{n}\|^{2}_{\y}=|v_{n}^{+}(\cdot, 0)|_{\2}^{\2}+ o_{n}(1).
\end{align*}
Because $(v_{n})$ is bounded in $Y_{\mu}$, we may assume that there exists $\ell\geq 0$ such that
\begin{align}\label{HB21}
\|v_{n}\|^{2}_{\y}\ri \ell  \quad \mbox{ and }  \quad |v_{n}^{+}(\cdot, 0)|_{\2}^{\2}\ri \ell.
\end{align} 
Suppose by contradiction that $\ell>0$. By virtue of $L_{\mu}(v_{n})=d_{\mu}+o_{n}(1)$, \eqref{HB20}, and \eqref{HB21}, we have
\begin{align*}
d_{\mu}+o_{n}(1)&= L_{\mu}(v_{n})=\frac{1}{2}\|v_{n}\|^{2}_{Y_{\mu}}-\int_{\R^{N}}  F(v_{n}(x, 0))\, dx-\frac{1}{\2} |v_{n}^{+}(\cdot, 0)|_{\2}^{\2} \\
&=\frac{\ell}{2}-\frac{\ell}{\2}+o_{n}(1)=\frac{s}{N}\ell+o_{n}(1),
\end{align*}
i.e., 
\begin{align}\label{HB22}
\ell=\frac{N}{s}d_{\mu}. 
\end{align}
On the other hand, by Theorem \ref{Sembedding} and \eqref{DUPRESS},  
\begin{align*}
\zeta S_{*}|v^{+}_{n}(\cdot, 0)|_{\2}^{2}\leq \zeta S_{*} |v_{n}(\cdot, 0)|_{\2}^{2}\leq \zeta \|v_{n}\|^{2}_{\x}\leq \|v_{n}\|_{Y_{\mu}}^{2},
\end{align*}
and passing to the limit as $n\ri \infty$, we arrive at
\begin{align*}
\zeta S_{*}\ell^{\frac{2}{\2}}\leq \ell.
\end{align*}
Taking \eqref{HB22} into account, we obtain that $d_{\mu}\geq \frac{s}{N} (\zeta S_{*})^{\frac{N}{2s}}$, which is impossible in view of \eqref{BREZISN}. Then, $\ell=0$, and this completes the proof.
\end{proof}

\noindent
Now we are ready to provide the main result of this section.
\begin{thm}\label{EGS}
Let $\mu>-V_{1}(1+\frac{1}{\kappa})$.
Then, \eqref{AIKCP} has a ground state solution.
\end{thm}
\begin{proof}
Since $L_{\mu}$ has a mountain pass geometry \cite{AR}, we can find a Palais-Smale sequence $(v_{n})\subset \y$ at the level $d_{\mu}$. Hence, $(v_{n})$ is bounded in $\y$, and so, up to a subsequence, we may suppose that there exists $v\in \y$ such that $v_{n}\rightharpoonup v$ in $\y$. Using the growth assumptions on $f$ and the density of $C^{\infty}_{c}(\overline{\R^{N+1}_{+}})$ in $\x$, it is standard to verify that $\langle L'_{\mu}(v), \varphi\rangle =0$ for all $\varphi\in \y$. If $v\equiv 0$, we can apply Lemma \ref{Lions2} to deduce that for some sequence $(z_{n})\subset \R^{N}$, $\bar{v}_{n}(x, y):=v_{n}(x+z_{n}, y)$ is a bounded Palais-Smale sequence at the level $d_{\mu}$ and $\bar{v}_{n}\rightharpoonup \bar{v}\not\equiv 0$ in $Y_{\mu}$. Thus, $\bar{v}\in \mathcal{M}_{\mu}$.
Moreover, by $(f_3)$ and Fatou's lemma, 
\begin{align*}
d_{\mu}&\leq L_{\mu}(\bar{v})\\
&=L_{\mu}(\bar{v})-\frac{1}{\theta} \langle L'_{\mu}(\bar{v}), \bar{v} \rangle\\
&=\left(\frac{1}{2}-\frac{1}{\theta}\right)\|\bar{v}\|^{2}_{Y_{\mu}}+\frac{1}{\theta} \int_{\R^{N}} [\bar{v}(x, 0)f(\bar{v}(x, 0))-\theta F(\bar{v}(x, 0))]\, dx+\left(\frac{1}{\theta}-\frac{1}{\2} \right)|\bar{v}^{+}(\cdot, 0)|^{\2}_{\2} \\
&\leq \liminf_{n\ri \infty} \left[ \left(\frac{1}{2}-\frac{1}{\theta}\right)\|\bar{v}_{n}\|^{2}_{Y_{\mu}}+\frac{1}{\theta} \int_{\R^{N}} [\bar{v}_{n}(x, 0)f(\bar{v}_{n}(x, 0))-\theta F(\bar{v}_{n}(x, 0))]\, dx+\left(\frac{1}{\theta}-\frac{1}{\2} \right)|\bar{v}_{n}^{+}(\cdot, 0)|^{\2}_{\2} \right] \\
&=\liminf_{n\ri \infty} \left[L_{\mu}(\bar{v}_{n})-\frac{1}{\theta} \langle L'_{\mu}(\bar{v}_{n}), \bar{v}_{n} \rangle\right]=d_{\mu},
\end{align*}
and so $L_{\mu}(\bar{v})=d_{\mu}$.
When $v\not\equiv 0$, as before, we can prove that $v$ is a ground state solution to \eqref{AEP}.
Consequently,
for each $\mu>-V_{1}(1+\frac{1}{\kappa})$,
there exists a ground state solution $w=w_{\mu}\in \y\setminus\{0\}$ such that
\begin{align*}
L_{\mu}(w)=d_{\mu} \quad \mbox{ and } \quad L'_{\mu}(w)=0.
\end{align*} 
As $f(t)=0$ for $t\leq 0$, it follows from $\langle L'_{\mu}(w), w^{-} \rangle=0$  that $w\geq 0$ in $\R^{N+1}_{+}$ and $w\not\equiv 0$. 
A standard Moser iteration argument (see, for instance, \cite[Lemma 4.4.1]{AmbrosioBOOK}, \cite[Lemma 4.1]{AJDE} or Lemma \ref{moser}) shows that $w(\cdot, 0)\in L^{r}(\R^{N})$ for all $r\in [2, \infty]$. 
According to \cite[Corollary 3]{ADCDS}, we know that $w(\cdot, 0)\in C^{0, \alpha}(\R^{N})$ for some $\alpha\in (0, 1)$. 
By the weak Harnack inequality \cite[Proposition 2]{FF}, we obtain that $w(\cdot, 0)>0$ in $\R^{N}$.
\end{proof}

\noindent
We conclude this section by establishing an important relation between $c_{\e}$ and $d_{V(0)}=d_{-V_{0}}$ 
(note that $V(0)=-V_{0}>-V_{1}(1+\frac{1}{\kappa})$ thanks to $0\in M$ and $V_{1}-V_{0}\geq 0>-\frac{V_{1}}{\kappa}$).
\begin{lem}\label{lem2.3AM}
The numbers $c_{\e}$ and $d_{V(0)}$ verify the following inequality:
$$
\limsup_{\e\ri 0} c_{\e}\leq d_{V(0)}<\frac{s}{N}(\zeta S_{*})^{\frac{N}{2s}}.
$$
\end{lem}
\begin{proof}
In light of Theorem \ref{EGS}, there exists a ground state solution $w$ to \eqref{AEP} with $\mu=V(0)$. 
Take $\eta\in C^{\infty}_{c}(\R)$ such that $0\leq \eta\leq 1$, $\eta=1$ in $[-1, 1]$ and $\eta=0$ in $\R\setminus (-2, 2)$. Suppose that $B_{2}\subset \Lambda$. Define $w_{\e}(x, y):=\eta(\e |(x, y)|) w(x, y)$ and note that $\supp(w_{\e}(\cdot, 0))\subset \Lambda_{\e}$. 
It is easy to see that $w_{\e}\ri w$ in $\x$ and that $L_{V(0)}(w_{\e})\ri L_{V(0)}(w)$ as $\e\ri 0$.
On the other hand, by the definition of $c_{\e}$, we have
\begin{align}\label{15ADOM}
c_{\e}\leq \max_{t\geq 0} J_{\e}(t w_{\e})=J_{\e}(t_{\e} w_{\e})=\frac{t^{2}_{\e}}{2} \|w_{\e}\|^{2}_{\e}-\int_{\R^{N}} \left[F(t_{\e} w_{\e}(x, 0))+\frac{1}{\2}(t_{\e}w_{\e}(x, 0))^{\2}\right]\, dx,
\end{align}
for some $t_{\e}>0$. Using $w\in \mathcal{M}_{V(0)}$ and $(f_4)$, we deduce that $t_{\e}\ri 1$ as $\e\ri 0$.
Observe that
\begin{align}\label{16ADOM}
J_{\e}(t_{\e}w_{\e})=L_{V(0)}(t_{\e}w_{\e})+\frac{t^{2}_{\e}}{2}\int_{\R^{N}} (V_{\e}(x)-V(0)) w_{\e}^{2}(x, 0)\, dx.
\end{align}
Then, since $V_{\e}(x)$ is bounded on the support of $w_{\e}(\cdot, 0)$ and $V_{\e}(x)\ri V(0)$ as $\e\ri 0$, we can exploit the dominated convergence theorem, \eqref{BREZISN}, \eqref{15ADOM}, and \eqref{16ADOM} to reach the desired conclusion.
\end{proof}

\begin{remark}\label{REMARKTERMOSIFONE}
From $(V_1)$ and $(g_2)$, we derive that $c_{\e}\geq d_{-V_{1}}>0$ for all $\e>0$.
\end{remark}

\section{Proof of Theorem \ref{thm1}}
This section is devoted to the proof of Theorem \ref{thm1}. Let us recall that, by Theorem \ref{lemAM2.8},
for all $\e>0$ there exists a nonnegative mountain pass solution $v_{\e}$ to \eqref{MEP}. We begin with a useful result.
\begin{lem}\label{lem2.4AM}
There exist $r, \beta, \e^{*}>0$ and $(y_{\e})\subset \R^{N}$ such that 
$$
\int_{B_{r}(y_{\e})} v_{\e}^{2}(x, 0)\, dx\geq \beta, \quad \mbox{ for all } \e\in (0, \e^{*}).
$$
\end{lem}
\begin{proof}
On account of \eqref{AM2.8} and the growth conditions on $f$, we can find $\alpha>0$, independent of $\e>0$, such that 
\begin{equation}\label{contradiction}
\|v_{\e}\|_{\e}^{2}\geq \alpha, \quad \mbox{ for all } \e>0.
\end{equation}
Let $(\e_{n})\subset (0, \infty)$ be such that $\e_{n}\ri 0$.
Assume, by contradiction, that there exists $r>0$ such that
$$
\lim_{n\ri \infty}\sup_{y\in \R^{N}} \int_{B_{r}(y)} v_{\e_{n}}^{2}(x, 0)\, dx=0.
$$
By Lemma \ref{Lions}, we know that $v_{\e_{n}}(\cdot, 0)\ri 0$ in $L^{r}(\R^{N})$ for all $r\in (2, \2)$. Hence, \eqref{AM2.8} and the growth assumptions on $f$ yield
\begin{align*}
\int_{\mathbb{R}^{N}} F(v_{\e_n}(x, 0)) \,dx=\int_{\mathbb{R}^{N}} f(v_{\e_{n}}(x, 0))v_{\e_{n}}(x, 0) \,dx=o_{n}(1).
\end{align*}
This implies that 
\begin{align}\label{2.11HZAMPA}
\int_{\mathbb{R}^{N}}  G_{\e_{n}}(x, v_{\e_{n}}(x, 0)) \,dx\leq \frac{1}{2^{*}_{s}} \int_{\Lambda_{\e_{n}}\cup \{v_{\e_{n}(\cdot, 0)}\leq a\}} v^{\2}_{\e_{n}}(x, 0) \,dx+\frac{V_{1}}{2\kappa} \int_{\Lambda^{c}_{\e_{n}}\cap \{v_{\e_{n}}(\cdot, 0)> a\}} v_{\e_{n}}^{2}(x, 0) \,dx+o_{n}(1)
\end{align}
and
\begin{align}\label{2.12HZAMPA}
\int_{\mathbb{R}^{N}}  g_{\e_{n}}(x, v_{\e_{n}}(x, 0))v_{\e_n}(x, 0) \,dx= \int_{\Lambda_{\e_{n}}\cup \{v_{\e_{n}}(\cdot, 0)\leq a\}} v^{\2}_{\e_{n}}(x, 0) \,dx+\frac{V_{1}}{\kappa} \int_{\Lambda^{c}_{\e_{n}}\cap \{v_{\e_{n}}(\cdot, 0)> a\}} v_{\e_{n}}^{2}(x, 0) \,dx+o_{n}(1).
\end{align}
Because of $\langle J'_{\e_{n}}(v_{\e_{n}}), v_{\e_{n}}\rangle=0$ and \eqref{2.12HZAMPA}, we obtain 
\begin{align}\label{2.13HZAMPA}
\|v_{\e_{n}}\|^{2}_{\e_{n}}-\frac{V_{1}}{\kappa} \int_{\Lambda^{c}_{\e_{n}}\cap \{v_{\e_{n}}(\cdot, 0)> a\}} v_{\e_{n}}^{2}(x, 0) \,dx=\int_{\Lambda_{\e_{n}}\cup \{v_{\e_{n}}(\cdot, 0)\leq a\}} v^{\2}_{\e_{n}}(x, 0) \,dx+o_{n}(1).
\end{align}
Let $\ell\geq 0$ be such that 
$$
\|v_{\e_{n}}\|^{2}_{\e_{n}}-\frac{V_{1}}{\kappa} \int_{\Lambda^{c}_{\e_{n}}\cap \{v_{\e_{n}}(\cdot, 0)> a\}} v_{\e_{n}}^{2}(x, 0) \,dx\rightarrow \ell.
$$
It is clear that $\ell>0$; otherwise, $\|v_{\e_{n}}\|_{\e_{n}}\rightarrow 0$, and this is impossible due to \eqref{contradiction} (alternatively, one can observe that $\|v_{\e_{n}}\|_{\e_{n}}\rightarrow 0$ yields $c_{\e_{n}}=J_{\e_{n}}(v_{\e_{n}})\ri 0$, which is a contradiction since Remark \ref{REMARKTERMOSIFONE} ensures that $c_{\e_{n}}\geq d_{-V_{1}}>0$ for all $n\in \mathbb{N}$).
From \eqref{2.13HZAMPA}, we derive that 
$$
\int_{\Lambda_{\e_{n}}\cup \{v_{\e_{n}}(\cdot, 0)\leq a\}} v^{\2}_{\e_{n}}(x, 0) dx\rightarrow \ell.
$$
Using $J_{\e_{n}}(v_{\e_{n}})-\frac{1}{ 2^{*}_{s}}\langle J'_{\e_{n}}(v_{\e_{n}}), v_{\e_{n}}\rangle=c_{\e_{n}}$, \eqref{2.11HZAMPA}, and \eqref{2.12HZAMPA}, we arrive at
\begin{equation}\label{MOLICA1AMPA}
\frac{s}{N} \ell\leq \liminf_{n\rightarrow \infty} c_{\e_{n}}.
\end{equation}
On the other hand, noting that 
\begin{align*}
\int_{\R^{N}} V_{\e_{n}}(x)v_{\e_{n}}^{2}(x, 0) \,dx-\frac{V_{1}}{\kappa} \int_{\Lambda^{c}_{\e_{n}}\cap \{v_{\e_{n}}(\cdot, 0)> a\}} v_{\e_{n}}^{2}(x, 0) \,dx&\geq -\left(1+\frac{1}{\kappa}\right)V_{1}\int_{\R^{N}} v^{2}_{\e_{n}}(x, 0)\, dx\\
&\geq -\left(1+\frac{1}{\kappa}\right)\frac{V_{1}}{m^{2s}}\| v_{\e_{n}} \|^{2}_{\x},
\end{align*}
by the definitions of $S_{*}$ and $\zeta$, we see that
$$
\|v_{\e_{n}}\|^{2}_{\e_{n}}-\frac{V_{1}}{\kappa} \int_{\Lambda^{c}_{\e_{n}}\cap \{v_{\e_{n}}(\cdot, 0)> a\}} v_{\e_{n}}^{2}(x, 0) \,dx\geq \zeta  S_{*} \left(\int_{\Lambda_{\e_{n}}\cup \{v_{\e_{n}}(\cdot, 0)\leq a\}} v^{\2}_{\e_{n}}(x, 0) \,dx  \right)^{\frac{2}{ 2^{*}_{s}}},
$$
and letting $n\rightarrow \infty$, we infer that 
\begin{equation}\label{MOLICA2AMPA}
\ell\geq \zeta S_{*}\ell^{\frac{2}{ 2^{*}_{s}}}.
\end{equation}
Combining \eqref{MOLICA1AMPA} and \eqref{MOLICA2AMPA} with the fact that $\ell>0$, we obtain 
\begin{align*}
\liminf_{n\rightarrow \infty} c_{\e_{n}}\geq \frac{s}{N} (\zeta S_{*})^{\frac{N}{2s}},
\end{align*}
which contradicts Lemma \ref{lem2.3AM}.
\end{proof}

\begin{lem}\label{lem2.5AM}
For each sequence $(\e_{n})$ such that $\e_{n}\ri 0$, consider the sequence $(y_{\e_{n}})\subset \R^{N}$ given in Lemma \ref{lem2.4AM}. Set $w_{n}(x, y):=v_{\e_{n}}(x+y_{\e_{n}}, y)$. Then, there exist a subsequence of $(w_{n})$, still denoted by itself, and $w\in \x\setminus \{0\}$ such that
\begin{align*}
w_{n}\ri w  \quad \mbox{ in } \x.
\end{align*}
Moreover, there exists $x_{0}\in \Lambda$ such that 
\begin{align*}
\e_{n}y_{\e_{n}}\ri x_{0} \quad \mbox{ and } \quad V(x_{0})=-V_{0}.
\end{align*}
\end{lem}
\begin{proof}
Hereafter, we denote by $(y_{n})$ and $(v_{n})$, the sequences $(y_{\e_n})$ and $(v_{\e_n})$, respectively. Exploiting \eqref{AM2.8}, Lemma \ref{lem2.3AM} and \eqref{equivalent}, we can argue as in the proof of Lemma \ref{boundPS} to deduce that $(w_{n})$ is bounded in $\x$. Hence, up to a subsequence, there exists $w\in \x\setminus \{0\}$ such that, as $n\rightarrow \infty$,  
\begin{equation}\label{2.15AM}
\left\{
\begin{array}{ll}
w_{n}\rightharpoonup w &\mbox{ in } X^{s}(\mathbb{R}^{N+1}_{+}), \\
w_{n}(\cdot, 0)\rightarrow w(\cdot, 0) &\mbox{ in } L^{r}_{loc}(\mathbb{R}^{N}), \,\, \mbox{ for all } r\in [1, 2^{*}_{s}), \\
w_{n}(\cdot, 0)\rightarrow w(\cdot, 0) &\mbox{ a.e. in } \mathbb{R}^{N},
\end{array}
\right.
\end{equation}
and
\begin{align}\label{2.16AM}
\int_{B_{r}} w^{2}(x, 0)\, dx\geq \beta>0,
\end{align}
where we have used Lemma \ref{lem2.4AM}.
Next, we will show that $(\e_{n}y_{n})$ is bounded in $\R^N$. To this end, it suffices to prove that
\begin{align}\label{Claim1}
{\rm dist}(\e_{n}y_{n}, \overline{\Lambda})\ri 0 \quad \mbox{ as } n\ri \infty.
\end{align}
In fact, if \eqref{Claim1} does not hold, there exist $\delta>0$ and a subsequence of $(\e_{n}y_{n})$, still denoted by itself, such that 
$$
{\rm dist}(\e_{n}y_{n}, \overline{\Lambda})\geq \delta, \quad  \mbox{ for all } n\in \mathbb{N}.
$$
Consequently, we can find $R>0$ such that $B_{R}(\e_{n}y_{n})\subset \Lambda^{c}$ for all $n\in \mathbb{N}$. Since $w\geq 0$, it follows from the definition of $\x$  that there exists a nonnegative sequence $(\psi_{j})\subset \x$ such that $\psi_{j}$ has compact support in $\overline{\mathbb{R}^{N+1}_{+}}$ and $\psi_{j}\ri w$ in $\x$ as $j\ri \infty$. Fix $j\in \mathbb{N}$. 
Inserting $\psi_{j}$ into the relation $\langle J'_{\e_{n}}(v_{n}), \phi\rangle=0$ for all $\phi\in X_{\e_{n}}$, we can write 
\begin{align}\label{2.17AM}
&\iint_{\R^{N+1}_{+}} y^{1-2s} (\nabla w_{n}\cdot\nabla \psi_{j}+m^{2}w_{n}\psi_{j})\, dx dy+\int_{\R^{N}} V_{\e_{n}}(x+y_{n}) w_{n}(x, 0) \psi_{j}(x, 0)\, dx  \nonumber\\
&=\int_{\R^{N}} g_{\e_{n}}(x+y_{n}, w_{n}(x, 0))\psi_{j}(x, 0)\, dx.
\end{align}
Note that, by the properties of $g_{\e}$, 
\begin{align*}
&\int_{\R^{N}} g_{\e_{n}}(x+y_{n}, w_{n}(x, 0))\psi_{j}(x, 0)\, dx\\
&= \int_{B_{\frac{R}{\e_{n}}}} g_{\e_{n}}(x+y_{n}, w_{n}(x, 0))\psi_{j}(x, 0)\, dx +\int_{B^{c}_{\frac{R}{\e_{n}}}} g_{\e_{n}}(x+y_{n}, w_{n}(x, 0))\psi_{j}(x, 0)\, dx \\
&\leq \frac{V_{1}}{\kappa} \int_{B_{\frac{R}{\e_{n}}}} w_{n}(x, 0) \psi_{j}(x, 0)\, dx+\int_{B^{c}_{\frac{R}{\e_{n}}}} [f(w_{n}(x, 0))+w^{\2-1}_{n}(x, 0)]\psi_{j}(x, 0)\, dx,
\end{align*}
which combined with $(V_1)$ and \eqref{2.17AM} yields
\begin{align*}
&\iint_{\R^{N+1}_{+}} y^{1-2s} (\nabla w_{n}\cdot\nabla \psi_{j}+m^{2}w_{n}\psi_{j})\, dx dy-V_{1}\left(1+\frac{1}{\kappa}\right) \int_{\R^{N}} w_{n}(x, 0) \psi_{j}(x, 0)\, dx  \\
&\leq \int_{B_{\frac{R}{\e_{n}}}^{c}} [f(w_{n}(x, 0))+w^{\2-1}_{n}(x, 0)]\psi_{j}(x, 0)\, dx.
\end{align*}
Recalling that $\psi_{j}$ has compact support and exploiting $\e_{n}\ri 0$, \eqref{2.15AM}, Theorem \ref{Sembedding}, the growth assumptions on $f$, we have that, as $n\ri \infty$,
\begin{align*}
\int_{B_{\frac{R}{\e_{n}}}^{c}} [f(w_{n}(x, 0))+w^{\2-1}_{n}(x, 0)]\psi_{j}(x, 0)\, dx\ri 0
\end{align*} 
and
\begin{align*}
&\iint_{\R^{N+1}_{+}} y^{1-2s} (\nabla w_{n}\cdot\nabla \psi_{j}+m^{2}w_{n}\psi_{j})\, dx dy-V_{1}\left(1+\frac{1}{\kappa}\right) \int_{\R^{N}} w_{n}(x, 0) \psi_{j}(x, 0)\, dx   \\
&\ri \iint_{\R^{N+1}_{+}} y^{1-2s} (\nabla w\cdot\nabla \psi_{j}+m^{2}w\psi_{j})\, dx dy-V_{1}\left(1+\frac{1}{\kappa}\right) \int_{\R^{N}} w(x, 0) \psi_{j}(x, 0)\, dx.
\end{align*}
Hence, for all $j\in \mathbb{N}$,
\begin{align*}
\iint_{\R^{N+1}_{+}} y^{1-2s} (\nabla w\cdot\nabla \psi_{j}+m^{2}w\psi_{j})\, dx dy-V_{1}\left(1+\frac{1}{\kappa}\right) \int_{\R^{N}} w(x, 0) \psi_{j}(x, 0)\, dx\leq 0,
\end{align*}
and letting $j\ri \infty$, we obtain
\begin{align*}
\|w\|^{2}_{\x}-V_{1}\left(1+\frac{1}{\kappa}\right) |w(\cdot, 0)|^{2}_{2} dx\leq 0.
\end{align*}
Using \eqref{m-ineq} and $\kappa>\frac{V_{1}}{m^{2s}-V_{1}}$, we arrive at 
$$
0\leq \left(1-\frac{V_{1}}{m^{2s}} \left(1+\frac{1}{\kappa}\right) \right)\|w\|^{2}_{\x}\leq 0,
$$
which contradicts \eqref{2.16AM}. By virtue of \eqref{Claim1}, there exist a subsequence of $(\e_{n}y_{n})$, still denoted by itself, and $x_{0}\in \overline{\Lambda}$ such that $\e_{n}y_{n}\ri x_{0}$ as $n\ri \infty$. Next, we claim that $x_{0}\in \Lambda$.

By $(g_2)$ and \eqref{2.17AM}, we have that
\begin{align*}
&\iint_{\R^{N+1}_{+}} y^{1-2s} (\nabla w_{n}\cdot\nabla \psi_{j}+m^{2}w_{n}\psi_{j})\, dx dy+\int_{\R^{N}} V_{\e_{n}}(x+y_{n}) w_{n}(x, 0) \psi_{j}(x, 0)\, dx \\
&\leq \int_{\R^{N}} [f(w_{n}(x, 0))+w^{\2-1}_{n}(x, 0)] \psi_{j}(x, 0)\, dx,
\end{align*}
and taking the limit as $n\ri \infty$, it follows from \eqref{2.15AM}, Theorem \ref{Sembedding} and the continuity of $V$ that
\begin{align*}
&\iint_{\R^{N+1}_{+}} y^{1-2s} (\nabla w\cdot\nabla \psi_{j}+m^{2}w\psi_{j})\, dx dy+\int_{\R^{N}} V(x_{0}) w(x, 0) \psi_{j}(x, 0)\, dx \\
&\leq \int_{\R^{N}} [f(w(x, 0))+w^{\2-1}(x, 0)] \psi_{j}(x, 0)\, dx.
\end{align*}
Letting $j\ri \infty$, we find 
\begin{align*}
\iint_{\R^{N+1}_{+}} y^{1-2s} (|\nabla w|^{2}+m^{2}w^{2})\, dx dy+\int_{\R^{N}} V(x_{0}) w^{2}(x, 0)\, dx \leq \int_{\R^{N}} [f(w(x, 0)) w(x, 0)+w^{\2}(x, 0)]\, dx.
\end{align*}
Hence, there is $t_{1}\in (0, 1)$ such that $t_{1}w\in \mathcal{M}_{V(x_{0})}$. Thus, by Lemma \ref{lem2.3AM}, we see that
\begin{align*}
d_{V(x_{0})}\leq L_{V(x_{0})}(t_{1}w)\leq \liminf_{n\ri \infty} J_{\e_{n}}(v_{n})=\liminf_{n\ri \infty} c_{\e_{n}}\leq d_{V(0)}.
\end{align*}
Therefore, $d_{V(x_{0})}\leq d_{V(0)}$, which implies that $V(x_{0})\leq V(0)=-V_{0}$.
Since $-V_{0}=\inf_{x\in \overline{\Lambda}} V(x)$, we deduce that $V(x_{0})=-V_{0}$. Moreover, by $(V_2)$, $x_{0}\notin \partial \Lambda$, and we can infer that $x_{0}\in \Lambda$.

Now, we aim to show that $w_{n}\ri w$ in $\x$ as $n\ri \infty$. For this purpose, for all $n\in \mathbb{N}$ and $x\in \R^{N}$, we 
set
$$
\tilde{\Lambda}_{n} := \frac{\Lambda - \e_{n}\tilde{y}_{n}}{\e_{n}}
$$ 
and
\begin{align*}
&\tilde{\chi}_{n}^{1}(x):= \left\{
\begin{array}{ll}
1, \, &\mbox{ if } x\in \tilde{\Lambda}_{n},\\
0, \, &\mbox{ if } x\in \tilde{\Lambda}^{c}_{n}, 
\end{array}
\right.\\
&\tilde{\chi}_{n}^{2}(x):= 1- \tilde{\chi}_{n}^{1}(x).
\end{align*}
Let us define the following functions for $x\in \R^{N}$ and $n\in \mathbb{N}$:
\begin{align*}
&h_{n}^{1}(x):= \left(\frac{1}{2}-\frac{1}{\theta}\right) (V_{\e_{n}}(x+ y_{n})+V_{1}) w^{2}_{n}(x, 0) \tilde{\chi}_{n}^{1}(x),\\
&h^{1}(x):=\left(\frac{1}{2}-\frac{1}{\theta}\right) (V(x_{0})+V_{1}) w^{2}(x, 0), \\
&h_{n}^{2}(x)\!\!:=\!\!\left[ \left(\frac{1}{2}-\frac{1}{\theta}\right) (V_{\e_{n}}(x+y_{n})+V_{1}) w^{2}_{n}(x, 0) + \frac{1}{\theta} g_{\e_{n}}(x+y_{n}, w_{n}(x, 0)) w_{n}(x, 0) - G_{\e_{n}}(x+y_{n}, w_{n}(x, 0))\right] \tilde{\chi}_{n}^{2}(x), \\
&h_{n}^{3}(x):= \left(\frac{1}{\theta} g_{\e_{n}}(x+y_{n}, w_{n}(x, 0)) w_{n}(x, 0) - G_{\e_{n}}(x+y_{n}, w_{n}(x, 0))\right) \tilde{\chi}_{n}^{1}(x) \\
&\quad \quad \quad \! \! \!=\left[\frac{1}{\theta} \left(\left(f(w_{n}(x, 0))w_{n}(x, 0)+(w_{n}(x, 0))^{\2}\right)- \left(F(w_{n}(x, 0))+\frac{1}{\2}(w_{n}(x, 0))^{\2}\right)\right) \right] \tilde{\chi}_{n}^{1}(x),  \\
&h^{3}(x):= \frac{1}{\theta} \left(f(w(x, 0))w(x, 0)+(w(x, 0))^{\2}\right)- \left(F(w(x, 0))+\frac{1}{\2}(w(x, 0))^{\2}\right). 
\end{align*}
In view of $(f_3)$, $(g_3)$, $(V_1)$, and our choice of $\kappa$, the aforementioned functions are nonnegative in $\R^{N}$.
Furthermore, using the following relations of limits, as $n\ri \infty$, 
\begin{align*}
&w_{n}(x, 0) \ri w(x, 0)\quad \mbox{ for a.e. } x\in \R^{N}, \\
&\e_{n}y_{n}\ri x_{0}\in \Lambda,
\end{align*}
we deduce that, as $n\ri \infty$,
\begin{align*}
&\tilde{\chi}_{n}^{1}(x)\ri 1, \, h_{n}^{1}(x)\ri h^{1}(x), \, h_{n}^{2}(x)\ri 0, \, \mbox{ and } \, h_{n}^{3}(x)\ri h^{3}(x) \, \mbox{ for a.e. } x\in \R^{N}. 
\end{align*}
Then, by a direct computation,   
\begin{align*}
d_{V(0)} &\geq \limsup_{n\ri \infty} c_{\e_{n}} = \limsup_{n\ri \infty} \left( J_{\e_{n}}(v_{n}) - \frac{1}{\theta} \langle J'_{\e_{n}}(v_{n}), v_{n}\rangle \right)\\
&\geq \limsup_{n\ri \infty} \left\{\left(\frac{1}{2}-\frac{1}{\theta} \right)\left[\|w_{n}\|^{2}_{\x}-V_{1}|w_{n}(\cdot, 0)|^{2}_{2}\right]+ \int_{\R^{N}} (h_{n}^{1}+ h_{n}^{2}+ h_{n}^{3}) \, dx\right\}\\
&\geq \liminf_{n\ri \infty} \left\{\left(\frac{1}{2}-\frac{1}{\theta} \right)\left[\|w_{n}\|^{2}_{\x}-V_{1}|w_{n}(\cdot, 0)|^{2}_{2}\right]+ \int_{\R^{N}} (h_{n}^{1}+ h_{n}^{2}+ h_{n}^{3}) \, dx\right\} \\
&\geq \left(\frac{1}{2}-\frac{1}{\theta} \right)\left[\|w\|^{2}_{\x}-V_{1}|w(\cdot, 0)|^{2}_{2}\right]+ \int_{\R^{N}} (h^{1}+ h^{3}) \, dx
= d_{V(0)}.
\end{align*}
The aforementioned inequalities yield
\begin{align}\label{2.19AM}
\lim_{n\ri \infty}\|w_{n}\|^{2}_{\x}-V_{1}|w_{n}(\cdot, 0)|^{2}_{2}=\|w\|^{2}_{\x}-V_{1}|w(\cdot, 0)|^{2}_{2}
\end{align}
and 
\begin{align*}
h_{n}^{1}\ri h^{1}, \, h_{n}^{2}\ri 0 \, \mbox{ and }\, h_{n}^{3}\ri h^{3} \, \mbox{ in } \, L^{1}(\R^{N}). 
\end{align*}
Hence,
\begin{align*}
\lim_{n\ri \infty} \int_{\R^{N}} (V_{\e_{n}}(x+y_{n})+V_{1}) w^{2}_{n}(x, 0) \, dx = \int_{\R^{N}} (V(x_{0})+V_{1}) w^{2}(x, 0) \, dx, 
\end{align*}
and we obtain
\begin{align}\label{2.20AM}
\lim_{n\ri \infty} |w_{n}(\cdot, 0)|_{2}^{2}= |w(\cdot, 0)|_{2}^{2}. 
\end{align}
Combining \eqref{2.19AM} and \eqref{2.20AM}, and recalling that $\x$ is a Hilbert space, we conclude that
\begin{align*}
\|w_{n}-w\|_{\x}\ri 0 \quad \mbox{ as } n\ri \infty.
\end{align*}
\end{proof}

\noindent
Now we use a Moser iteration argument \cite{Moser} to establish a fundamental $L^{\infty}$-estimate.
\begin{lem}\label{moser}
Let $(w_{n})$ be the sequence defined as in Lemma \ref{lem2.5AM}. Then, $(w_{n}(\cdot, 0))\subset L^{\infty}(\R^{N})$, and there exists $C>0$ such that
\begin{align*}
|w_{n}(\cdot, 0)|_{\infty}\leq C, \quad \mbox{ for all } n\in \mathbb{N}.
\end{align*}
\end{lem}
\begin{proof}
It suffices to argue as in the proof of \cite[Lemma 4.1]{ADCDS}. 
However, for the reader's convenience, we provide a different proof here.
First, we observe that $w_{n}$ is a weak solution to
\begin{align}\label{traslato}
\left\{
\begin{array}{ll}
-\dive(y^{1-2s} \nabla w_{n})+m^{2}y^{1-2s}w_{n}=0 &\mbox{ in } \R^{N+1}_{+}, \\
\frac{\partial w_{n}}{\partial \nu^{1-2s}}=-V_{\e_{n}}(\cdot+y_{n})w_{n}(\cdot, 0)+g_{\e_{n}}(\cdot+y_{n}, w_{n}(\cdot, 0)) &\mbox{ on } \R^{N}.
\end{array}
\right.
\end{align}
For $\beta>1$ and $T>0$, we consider the following function:
\begin{equation*}
H(t):=
\left\{
\begin{array}{ll}
0, & \mbox{ if } t\leq 0, \\
t^{\beta}, & \mbox{ if } 0<t<T, \\
\beta T^{\beta-1}(t-T)+T^{\beta}, & \mbox{ if } t\geq T. \\
\end{array}
\right.
\end{equation*}
Note that $H:\R\ri \R$ is convex, nondecreasing, and Lipschitz continuous with Lipschitz constant $\beta T^{\beta-1}$. Define $\mathcal{L}(t):=\int_{0}^{t} (H'(\tau))^{2}\, d\tau$ for all $t\in \R$. Clearly, $\mathcal{L}\in C^{1}(\R)\cap W^{1, \infty}(\R)$ and $\mathcal{L}(0)=0$. Set
$$
\varphi_{n}(x, y):=\mathcal{L}(w_{n}(x, y))=\int_{0}^{w_{n}(x, y)} (H'(\tau))^{2}\, d\tau.
$$
Testing \eqref{traslato} with $\varphi_{n}$, we can write
\begin{align}\label{conto1JMP}
&\iint_{\mathbb{R}^{N+1}_{+}}  y^{1-2s} (\nabla w_{n}\cdot \nabla \varphi_{n} +m^{2} w_{n}\varphi_{n})  \, dxdy \nonumber\\
&= -\int_{\mathbb{R}^{N}} V_{\e_{n}}(x+y_{n}) w_{n}(x,0) \varphi_{n}(x,0) \,dx+ \int_{\mathbb{R}^{N}} g_{\e_{n}}(x+y_{n}, w_{n}(x, 0)) \varphi_{n}(x,0) \,dx. 
\end{align}
By $(V_1)$, \eqref{growthg}, and $\varphi_{n}\geq 0$, we see that
\begin{align*}
\left[-V_{\e_{n}}(x+y_{n}) w_{n}(x,0)+g_{\e_{n}}(x+y_{n}, w_{n}(\cdot, 0))\right] \varphi_{n}(x,0) &\leq \left[(V_{1}+\delta)w_{n}(x,0)+C_{\delta} w_{n}^{\2-1}(x,0) \right] \varphi_{n}(x,0) \\
&\leq C(1+w_{n}^{\2-1}(x,0))\varphi_{n}(x,0) \\
&\leq C(1+w_{n}^{\2-1}(x,0)) w_{n}(x, 0) (H'(w_{n}(x, 0)))^{2},
\end{align*}
where we have used $\varphi_{n}(x, 0)\leq w_{n}(x, 0) (H'(w_{n}(x, 0)))^{2}$.
Then, from \eqref{conto1JMP}, we derive that
\begin{align*}
\iint_{\mathbb{R}^{N+1}_{+}}  y^{1-2s} (\nabla w_{n}\cdot \nabla \varphi_{n} +m^{2} w_{n}\varphi_{n})  \, dxdy \leq C\int_{\mathbb{R}^{N}} (1+w_{n}^{\2-1}(x, 0))w_{n}(x,0) (H'(w_{n}(x, 0)))^{2} \,dx.
\end{align*}
Since $H(w_{n}(x, 0))H'(w_{n}(x, 0))\leq \beta^{2} w_{n}^{2\beta-1}(x, 0)$ and $w_{n}(x, 0)H'(w_{n}(x, 0))\leq \beta H(w_{n}(x, 0))$, we obtain
\begin{align}\label{AUTUORI1}
&\iint_{\mathbb{R}^{N+1}_{+}}  y^{1-2s} (\nabla w_{n}\cdot \nabla \varphi_{n} +m^{2} w_{n}\varphi_{n})  \, dxdy \nonumber\\
&\leq C\beta \int_{\mathbb{R}^{N}} (1+w_{n}^{\2-1}(x, 0)) H(w_{n}(x,0)) H'(w_{n}(x, 0)) \,dx \nonumber\\
&\leq C\beta \int_{\mathbb{R}^{N}} \left[\beta^{2}w_{n}^{2\beta-1}(x, 0)+\beta w_{n}^{\2-2}(x, 0) H^{2}(w_{n}(x,0))\right] \,dx \nonumber\\
&\leq C\beta^{3} \int_{\mathbb{R}^{N}} \left[w_{n}^{2\beta-1}(x, 0)+ w_{n}^{\2-2}(x, 0) H^{2}(w_{n}(x,0))\right] \,dx. 
\end{align}
On the other hand, using Theorem \ref{Sembedding} and $w_{n}, \varphi_{n}\geq 0$, we obtain
\begin{align}\label{AUTUORI2}
\iint_{\mathbb{R}^{N+1}_{+}}  y^{1-2s} (\nabla w_{n}\cdot \nabla \varphi_{n} +m^{2} w_{n}\varphi_{n})  \, dxdy&=\iint_{\mathbb{R}^{N+1}_{+}}  y^{1-2s} [|\nabla w_{n}|^{2} (H'(w_{n}))^{2} +m^{2} w_{n}\varphi_{n}] \, dxdy \nonumber\\
&\geq \iint_{\mathbb{R}^{N+1}_{+}}  y^{1-2s} |\nabla w_{n}|^{2} (H'(w_{n}))^{2}   \, dxdy \nonumber\\
&=\iint_{\mathbb{R}^{N+1}_{+}}  y^{1-2s} |\nabla H(w_{n})|^{2}   \, dxdy \nonumber\\
&\geq C \left( \int_{\R^{N}} |H(w_{n}(x, 0))|^{\2} \, dx\right)^{\frac{2}{\2}}.
\end{align}
Combining \eqref{AUTUORI1} and \eqref{AUTUORI2}, we find
\begin{align}\label{Me1}
 \left( \int_{\R^{N}} |H(w_{n}(x, 0))|^{\2} \, dx\right)^{\frac{2}{\2}} \leq C \beta^{3} \int_{\mathbb{R}^{N}} \left[w_{n}^{2\beta-1}(x, 0)+ w_{n}^{\2-2}(x, 0) H^{2}(w_{n}(x,0))\right] \,dx,
\end{align}
where $C>0$ is independent of $\beta$ and $T$. We stress that the last integral in \eqref{Me1} is well defined for every $T>0$ in the definition of $H$.
Now we choose $\beta$ in (\ref{Me1}) such that $2\beta-1=2^{*}_{s}$, and we name it $\beta_{1}$, i.e.,
\begin{equation}\label{Me2}
\beta_{1}:=\frac{2^{*}_{s}+1}{2}.
\end{equation}
Let $R>0$ to be fixed later. Concerning the last integral in \eqref{Me1}, applying the H\"older inequality with exponents $r:=\frac{2^{*}_{s}}{2}$ and $r':=\frac{2^{*}_{s}}{2^{*}_{s}-2}$, we see that
\begin{align}\label{Me3}
&\int_{\R^{N}} w_{n}^{\2-2}(x, 0) H^{2}(w_{n}(x,0))\, dx  \nonumber \\
&=  \int_{\{w_{n}(\cdot, 0)\leq R\}} w_{n}^{\2-2}(x, 0) H^{2}(w_{n}(x,0))\, dx 
+\int_{\{w_{n}(\cdot, 0)> R\}} w_{n}^{\2-2}(x, 0) H^{2}(w_{n}(x,0)) \, dx \nonumber\\
&\leq  R^{2^{*}_{s}-1} \int_{\{w_{n}(\cdot, 0)\leq R\}} \frac{H^{2}(w_{n}(x, 0))}{w_{n}(x, 0)} \, dx+\left(\int_{\R^{N}} |H(w_{n}(x, 0))|^{2^{*}_{s}} \, dx\right)^{\frac{2}{2^{*}_{s}}} \left(\int_{\{w_{n}(\cdot, 0)>R\}} w_{n}^{\2}(x, 0) \, dx\right)^{\frac{2^{*}_{s}-2}{2^{*}_{s}}}.
\end{align}
Because $(w_{n}(\cdot, 0))$ strongly converges in $L^{\2}(\R^{N})$ (by Lemma \ref{lem2.5AM}), we can take $R$ sufficiently large such that
\begin{align*}
\left(\int_{\{w_{n}(\cdot, 0)>R\}} w_{n}^{2^{*}_{s}}(x, 0) \, dx\right)^{\frac{2^{*}_{s}-2}{2^{*}_{s}}}\leq \frac{1}{2C \beta^{3}_{1}},
\end{align*}
where $C$ is the constant appearing in \eqref{Me1}.
This together with (\ref{Me1}), (\ref{Me2}), and (\ref{Me3}) yields
\begin{align}\label{Me4}
\left(\int_{\R^{N}} |H(w_{n}(x, 0))|^{\2} \, dx \right)^{\frac{2}{\2}}\leq 2 C\beta_{1}^{3}\left(\int_{\R^{N}} w_{n}^{2^{*}_{s}}(x, 0) \, dx+R^{2^{*}_{s}-1}\int_{\R^{N}} \frac{H^{2}(w_{n}(x, 0))}{w_{n}(x, 0)} \, dx \right).
\end{align}
In view of $H(w_{n}(x, 0))\leq w_{n}^{\beta_{1}}(x, 0)$ and (\ref{Me2}), and letting $T\rightarrow \infty$ in (\ref{Me4}), we obtain
\begin{align*}
\left(\int_{\R^{N}} w_{n}^{2^{*}_{s}\beta_{1}}(x, 0) \, dx\right)^{\frac{2}{2^{*}_{s}}}
\leq 2 C\beta_{1}^{3}\left(\int_{\R^{N}} w_{n}^{2^{*}_{s}}(x, 0) \, dx+R^{2^{*}_{s}-1}\int_{\R^{N}}  w_{n}^{2^{*}_{s}}(x, 0) \, dx\right),
\end{align*} 
which combined the boundedness of $(w_{n}(\cdot, 0))$ in $L^{\2}(\R^{N})$ implies
\begin{align}\label{Me5}
|w_{n}(\cdot, 0)|_{2^{*}_{s}\beta_{1}}\leq C', \quad  \mbox{ for all } n\in \mathbb{N}.
\end{align}
Now, we suppose $\beta>\beta_{1}$. Thus, using $H(w_{n}(x, 0))\leq w_{n}^{\beta}(x, 0)$ on the right-hand side of (\ref{Me1}) and passing to the limit as $T\rightarrow \infty$, we deduce that
\begin{align}\label{Me6}
&\left(\int_{\R^{N}} w_{n}^{2^{*}_{s}\beta}(x, 0) \, dx\right)^{\frac{2}{2^{*}_{s}}} \leq C\beta^{3} \left(\int_{\R^{N}} w_{n}^{2\beta-1}(x, 0) \, dx+\int_{\R^{N}} w_{n}^{2\beta+2^{*}_{s}-2}(x, 0) \, dx\right).
\end{align} 
Put
$$
a_{1}:=\frac{2^{*}_{s}(2^{*}_{s}-1)}{2(\beta-1)} \, \mbox{ and } \, a_{2}:=2\beta-1-a_{1}.
$$
Note that $0<a_{1}<\2$ and $a_{2}>0$ (since $\beta>\beta_{1}$).
Applying the Young inequality with exponents $r:=\frac{2^{*}_{s}}{a_{1}}$ and $r':=\frac{2^{*}_{s}}{2^{*}_{s}-a_{1}}$, we have that
\begin{align}\label{Me7}
\int_{\R^{N}} w_{n}^{2\beta-1}(x, 0) \, dx&\leq \frac{a_{1}}{2^{*}_{s}} \int_{\R^{N}} w_{n}^{2^{*}_{s}}(x, 0)\, dx+\frac{2^{*}_{s}-a_{1}}{2^{*}_{s}} \int_{\R^{N}} w_{n}^{\frac{2^{*}_{s} a_{2}}{2^{*}_{s}-a_{1}}}(x, 0) \, dx \nonumber\\
&\leq \int_{\R^{N}} w_{n}^{2^{*}_{s}}(x, 0)\, dx+\int_{\R^{N}} w_{n}^{2\beta+2^{*}_{s}-2}(x, 0)\, dx \nonumber\\
&\leq C\left(1+\int_{\R^{N}}  w_{n}^{2\beta+2^{*}_{s}-2}(x, 0)\, dx \right),
\end{align}
with $C>0$ independent of $\beta$ and $n\in \mathbb{N}$.
Combining (\ref{Me6}) and (\ref{Me7}), we get
\begin{align*}
\left(\int_{\R^{N}} w_{n}^{2^{*}_{s}\beta}(x, 0) \, dx\right)^{\frac{2}{2^{*}_{s}}}\leq C\beta^{3} \left(1+\int_{\R^{N}}  w_{n}^{2\beta+2^{*}_{s}-2}(x, 0)\, dx \right),
\end{align*}
with $C>0$ changing from line to line, but remaining independent of $\beta$ and $n\in \mathbb{N}$.
Therefore,
\begin{align}\label{Me9}
\left(1+\int_{\R^{N}} w_{n}^{2^{*}_{s}\beta}(x, 0) \, dx\right)^{\frac{1}{2^{*}_{s}(\beta-1)}}\leq (C\beta^{3})^{\frac{1}{2(\beta-1)}} \left(1+\int_{\R^{N}}  w_{n}^{2\beta+2^{*}_{s}-2}(x, 0)\, dx \right)^{\frac{1}{2(\beta-1)}}.
\end{align}
For $k\in \mathbb{N}$, we define $\beta_{k}$ inductively so that $2\beta_{k+1}+2^{*}_{s}-2=2^{*}_{s}\beta_{k}$, i.e.,
$$
\beta_{k+1}:=\left(\frac{2^{*}_{s}}{2}\right)^{k}(\beta_{1}-1)+1.
$$
Hence, from \eqref{Me9}, we obtain
\begin{align*}
\left(1+\int_{\R^{N}} w_{n}^{2^{*}_{s}\beta_{k+1}}(x, 0) \, dx\right)^{\frac{1}{2^{*}_{s}(\beta_{k+1}-1)}} \leq (C\beta^{3}_{k+1})^{\frac{1}{2(\beta_{k+1}-1)}} \left(1+\int_{\R^{N}}  w_{n}^{2^{*}_{s}\beta_{k}}(x, 0)\, dx \right)^{\frac{1}{2^{*}_{s}(\beta_{k}-1)}}.
\end{align*}
Setting
$$
A_{k, n}:=\left(1+\int_{\R^{N}} w_{n}^{2^{*}_{s}\beta_{k}}(x, 0) \, dx\right)^{\frac{1}{2^{*}_{s}(\beta_{k}-1)}}
$$
and
$$
C_{k+1}:=C\beta^{3}_{k+1},
$$
we can find a constant $C_{0}>0$ independent of $k$ such that 
$$
A_{k+1, n}\leq \prod_{j=2}^{k+1} C_{j}^{\frac{1}{2(\beta_{j}-1)}} A_{1, n}\leq C_{0} A_{1, n}, \quad \mbox{ for all } k, n\in \mathbb{N}.
$$
Since \eqref{Me5} implies that, for some $A_{0}>0$, $A_{1, n}\leq A_{0}$ for all $n\in \mathbb{N}$, we infer that 
$$
A_{k+1, n}\leq C_{0}A_{0}, \quad \mbox{ for all } k, n\in \mathbb{N}. 
$$
Consequently, letting $k\ri \infty$,
$$
|w_{n}(\cdot, 0)|_{\infty}\leq C, \quad \mbox{ for all } n\in \mathbb{N}.
$$
The proof of the lemma is now complete.
\end{proof}

\begin{remark}\label{REMARKFF1}
According to \cite[Proposition 3]{FF}, we have that $w_{n}\in C^{0, \alpha}_{loc}(\overline{\R^{N+1}_{+}})$ for all $n\in \mathbb{N}$.
\end{remark}

\medskip
\noindent
Now, we observe that $w_{n}(\cdot, 0)$ is a weak solution to
\begin{align*}
(-\Delta+m^{2})^{s}w_{n}(\cdot, 0)=-V_{\e_{n}}(\cdot+y_{n})w_{n}(\cdot, 0)+g_{\e_{n}}(\cdot+y_{n}, w_{n}(\cdot, 0)) \quad \mbox{ in } \R^{N}.
\end{align*}
Fix $\eta\in (0, m^{2s}-V_{1})$. Using $(V_{1})$ and \eqref{growthg}, we can deduce that $w_{n}(\cdot, 0)$ is a weak subsolution to
\begin{align}\label{COLUCCIa}
(-\Delta+m^{2})^{s}w_{n}(\cdot, 0)=(V_{1}+\eta) w_{n}(\cdot, 0)+C_{\eta}w_{n}^{\2-1}(\cdot, 0)=:\mu_{n} \quad \mbox{ in } \R^{N},
\end{align}
for some $C_{\eta}>0$. Note that $\mu_{n}\geq 0$ in $\R^{N}$. By Lemma \ref{moser} and interpolation in $L^{r}$ spaces, we know that, for all $r\in [2, \infty)$, 
$$
\mu_{n}\ri \mu:=(V_{1}+\eta) w(\cdot, 0)+C_{\eta}w^{\2-1}(\cdot, 0) \quad \mbox{ in } L^{r}(\R^{N}),
$$ 
and $|\mu_{n}|_{\infty}\leq C$ for all $n\in \mathbb{N}$. Let $z_{n}\in \h$ be the unique solution to 
\begin{align}\label{COLUCCIb}
(-\Delta+m^{2})^{s}z_{n}=\mu_{n} \quad \mbox{ in } \R^{N}.
\end{align}
Then, $z_{n}=\mathcal{G}_{2s, m}*\mu_{n}$, where $\mathcal{G}_{2s, m}(x):=(2\pi)^{-\frac{N}{2}}\mathcal{F}^{-1}((|k|^{2}+m^{2})^{-s})(x)$ is the Bessel kernel with parameter $m$ and $\mathcal{F}^{-1}$ denotes the inverse Fourier transform. 
By the scaling property of the Fourier transform, it follows that $\mathcal{G}_{2s, m}(x)=m^{N-2s}\mathcal{G}_{2s, 1}(mx)$.
Exploiting formula $(4.1)$ at pag. 416 in \cite{ArS} (where $\mathcal{G}_{2s, 1}$ is denoted by $G_{2s}$),
we can see that 
$$
\mathcal{G}_{2s, m}(x)=\frac{1}{2^{\frac{N+2s-2}{2}} \pi^{\frac{N}{2}}\Gamma(s)} m^{\frac{N-2s}{2}}K_{\frac{N-2s}{2}}(m|x|)|x|^{\frac{2s-N}{2}},
$$
and it satisfies the following properties (see pag. 416-417 in \cite{ArS} and pag. 132 in \cite{stein}, with $\alpha=2s$ and $m=1$):
\begin{compactenum}[$(\mathcal{G}1)$]
\item $\mathcal{G}_{2s, m}$ is positive, radially symmetric, and smooth in $\R^{N}\setminus\{0\}$,
\item $\mathcal{G}_{2s, m}(x)\leq C(\chi_{B_{2}}(x) |x|^{2s-N}+\chi_{B_{2}^{c}}(x) e^{-c|x|})$ for all $x\in \R^{N}$, for some $C, c>0$,
\item $\mathcal{G}_{2s, m}\in L^{r}(\R^{N})$ for all $r\in [1, \frac{N}{N-2s})$.
\end{compactenum}
In view of the aforementioned facts, we can prove the next crucial result.
\begin{lem}\label{lem2.6AM}
The sequence $(w_{n})$ satisfies $w_{n}(\cdot, 0)\ri 0$ as $|x|\ri \infty$ uniformly in $n\in \mathbb{N}$.
\end{lem}
\begin{proof}
We start by showing that $z_{n}(x)\ri 0$ as $|x|\ri \infty$ uniformly in $n\in \mathbb{N}$.
Note that, fixed $\delta\in (0, \frac{1}{2})$, it holds
\begin{align}\label{ASPRITE}
z_{n}(x)=(\mathcal{G}_{2s, m}*\mu_{n})(x)=\int_{B^{c}_{\frac{1}{\delta}}(x)} \mathcal{G}_{2s, m}(x-\xi)\mu_{n}(\xi) \, d\xi+\int_{B_{\frac{1}{\delta}}(x)} \mathcal{G}_{2s, m}(x-\xi)\mu_{n}(\xi) \, d\xi.
\end{align}
From $(\mathcal{G}1)$ and $(\mathcal{G}2)$, we derive that 
the first integral in \eqref{ASPRITE} can be estimated as follows: 
\begin{align}\label{A1SPRITE}
0\leq \int_{B^{c}_{\frac{1}{\delta}}(x)} \mathcal{G}_{2s, m}(x-\xi)\mu_{n}(\xi) \, d\xi &\leq C |\mu_{n}|_{\infty}\int_{B^{c}_{\frac{1}{\delta}}(x)} e^{-c|x-\xi|} \, d\xi \nonumber\\
&\leq C\int_{\frac{1}{\delta}}^{\infty} e^{-cr} r^{N-1} dr=:C A(\delta)\ri 0 \quad \mbox{ as } \delta\ri 0.
\end{align}
Concerning the second integral in \eqref{ASPRITE}, we observe that
\begin{align*}
0\leq \int_{B_{\frac{1}{\delta}}(x)} \mathcal{G}_{2s, m}(x-\xi)\mu_{n}(\xi) \, d\xi= \int_{B_{\frac{1}{\delta}}(x)} \mathcal{G}_{2s, m}(x-\xi)(\mu_{n}(\xi)-\mu(\xi)) \, d\xi+
\int_{B_{\frac{1}{\delta}}(x)} \mathcal{G}_{2s, m}(x-\xi)\mu(\xi) \, d\xi.
\end{align*}
Fix $q\in (1, \min\{\frac{N}{N-2s}, 2\})$ so that $q'>2$, where $q'$ is the conjugate exponent of $q$, i.e. $\frac{1}{q}+\frac{1}{q'}=1$.
By means of $(\mathcal{G}3)$ and H\"older's inequality, we have 
that
\begin{align*}
\int_{B_{\frac{1}{\delta}}(x)} \mathcal{G}_{2s, m}(x-\xi)\mu_{n}(\xi) \, d\xi\leq |\mathcal{G}_{2s, m}|_{q} |\mu_{n}-\mu|_{q'}+|\mathcal{G}_{2s, m}|_{q} |\mu|_{L^{q'}(B_{\frac{1}{\delta}}(x))}.
\end{align*}
Because $|\mu_{n}-\mu|_{q'}\rightarrow 0$ as $n\rightarrow \infty$ and $|\mu|_{L^{q'}(B_{\frac{1}{\delta}}(x))}\rightarrow 0$ as $|x|\rightarrow \infty$, there exist $R>0$ and $n_{0}\in \mathbb{N}$ such that
\begin{align}\label{A2SPRITE}
\int_{B_{\frac{1}{\delta}}(x)} \mathcal{G}_{2s, m}(x-\xi)\mu_{n}(\xi) \, d\xi\leq C\delta
\end{align}
for all $n\geq n_{0}$ and $|x|\geq R$. 
Putting together (\ref{A1SPRITE}) and (\ref{A2SPRITE}), we obtain that
\begin{align}\label{A3SPRITE}
\int_{\R^{N}} \mathcal{G}_{2s, m}(x-\xi)\mu_{n}(\xi) \, d\xi\leq C(A(\delta)+\delta)
\end{align}
for all $n\geq n_{0}$ and $|x|\geq R$.
On the other hand, for each $n\in \{1, \dots, n_{0}-1\}$, there exists $R_{n}>0$ such that $|\mu_{n}|_{L^{q'}(B_{\frac{1}{\delta}}(x))}<\delta$ as $|x|\geq R_{n}$. Thus, for $|x|\geq R_{n}$, 
\begin{align}\label{A4SPRITE}
\int_{\R^{N}} \mathcal{G}_{2s, m}(x-\xi)\mu_{n}(\xi) \, d\xi &\leq CA(\delta)+\int_{B_{\frac{1}{\delta}}(x)} \mathcal{G}_{2s, m}(x-\xi)\mu_{n}(\xi) \, d\xi  \nonumber\\
&\leq CA(\delta)+|\mathcal{G}_{2s, m}|_{q} |\mu_{n}|_{L^{q'}(B_{\frac{1}{\delta}}(x))} \nonumber\\
&\leq C(A(\delta)+\delta).
\end{align}
Hence, taking $\bar{R}:=\max\{R, R_{1}, \dots, R_{n_{0}-1} \}$, \eqref{A3SPRITE}, and \eqref{A4SPRITE}, ensure that
\begin{align*}
\int_{\R^{N}} \mathcal{G}_{2s, m}(x-\xi)\mu_{n}(\xi) \, d\xi\leq C(A(\delta)+\delta)
\end{align*}
for $|x|\geq \bar{R}$, uniformly in $n\in \mathbb{N}$. Letting $\delta\rightarrow 0$, we reach the desired result for $z_{n}$. 
In light of \eqref{COLUCCIa} and \eqref{COLUCCIb}, a simple comparison argument (see \cite[Theorem 4.3]{ADCDS} with $\Omega=\R^{N}$) shows that $0\leq w_{n}(\cdot, 0)\leq z_{n}$ in $\R^{N}$. Consequently, $w_{n}(\cdot, 0)\ri 0$ as $|x|\ri \infty$ uniformly in $n\in \mathbb{N}$. 
\end{proof}

\begin{remark} 
An alternative proof of Lemma \ref{lem2.6AM} can be established by using the elliptic estimates in \cite{FF} (see \cite[Lemma 4.2]{ADCDS}). For completeness, we give the details.
By Remark \ref{REMARKFF1}, we know that $w_{n}\in C^{0, \alpha}_{loc}(\overline{\R^{N+1}_{+}})$ for some $\alpha\in (0, 1)$ independent of $n$. Thanks to \eqref{weightedE} and $w_{n}\ri w$ in $\x$, we deduce that $w_{n}\ri w$ in $L^{2\gamma}(\R^{N+1}_{+}, y^{1-2s})$, where $\gamma=1+\frac{2}{N-2s}$. Since $w_{n}(\cdot, 0)\ri w(\cdot, 0)$ in $H^{s}(\R^{N})$, it follows from Lemma \ref{moser} that $w_{n}(\cdot, 0)\ri w(\cdot, 0)$ in $L^{r}(\R^{N})$ for all $r\in [2, \infty)$. Now, let $\bar{x}\in \R^{N}$ be fixed. Using $(V_1)$ and the growth assumptions on $g$, we see that $w_{n}$ is a weak subsolution to 
\begin{equation*}
\left\{
\begin{array}{ll}
-{\rm div}(y^{1-2s} \nabla w_{n})+m^{2}y^{1-2s}w_{n}=0  &\mbox{ in } Q_{1}(\bar{x},0):=B_{1}(\bar{x})\times (0,1), \\
\frac{\partial w_{n}}{\partial \nu^{1-2s}}= (V_{1}+\eta)w_{n}(\cdot, 0) +C_{\eta}w^{2^{*}_{s}-1}_{n}(\cdot, 0)  &\mbox{ on } B_{1}(\bar{x}),
\end{array}
\right.
\end{equation*}
where $\eta\in (0, m^{2s}-V_{1})$ is fixed. By \cite[Proposition 1]{FF}, we obtain
$$
\sup_{Q_{\frac{1}{2}}(\bar{x},0)} w_{n}(\cdot, 0)\leq C\left(\|w_{n}\|_{L^{2\gamma}(Q_{1}(\bar{x},0), y^{1-2s})}+|w_{n}^{2^{*}_{s}-1}(\cdot, 0)|_{L^{q_{0}}(B_{1}(\bar{x}))}\right), \, \mbox{ for all } n\in \mathbb{N},
$$
where $q_{0}>\frac{N}{2s}$ is fixed and $C>0$ is a constant 
independent of $n\in \mathbb{N}$ and $\bar{x}$. 
Exploiting the strong convergence of $(w_{n})$ in $L^{2\gamma}(\R^{N+1}_{+}, y^{1-2s})$ and of $(w_{n}(\cdot, 0))$ in $L^{q_{0}(\2-1)}(\R^{N})$, respectively, we infer that 
$w_{n}(\bar{x}, 0)\rightarrow 0$ as $|\bar{x}|\rightarrow \infty$ uniformly in $n\in \mathbb{N}$, and thus, the claim is proved. In this paper we prefer to give a proof based on the properties of the Bessel kernel $\mathcal{G}_{2s, m}$ that we believe to be useful for future references. 
\end{remark}

\noindent
We also have the following lemma.
\begin{lem}\label{lemDELTA}
There exists $\delta>0$ such that 
\begin{align*}
|w_{n}(\cdot, 0)|_{\infty}\geq \delta, \quad\mbox{ for all } n\in \mathbb{N}.
\end{align*} 
\end{lem}
\begin{proof}
By Lemma \ref{lem2.4AM}, there exist $r, \beta>0$ and $n_{0}\in \mathbb{N}$ such that
\begin{align*}
\int_{B_{r}} w_{n}^{2}(x, 0)\, dx\geq \beta, \quad \mbox{ for all } n\geq n_{0}.
\end{align*}
If by contradiction $|w_{n}(\cdot, 0)|_{\infty}\ri 0$ as $n\ri \infty$, then
\begin{align*}
0<\beta\leq \int_{B_{r}} w_{n}^{2}(x, 0)\, dx\leq |B_{r}| |w_{n}(\cdot, 0)|^{2}_{\infty}\ri 0 \quad \mbox{ as } n\ri \infty,
\end{align*} 
which of course is absurd.
\end{proof}

\noindent
We are now ready to give the proof of the main result of this paper.
\begin{proof}[Proof of Theorem \ref{thm1}]
By Lemma \ref{lem2.5AM}, up to a subsequence, there exist $(y_{n})\subset \mathbb{R}^{N}$ and $w\in \x\setminus\{0\}$ such that $w_{n}(x, y):=v_{n}(x+y_{n}, y)\rightarrow w$ in $\x$ and $\e_{n}y_{n}\rightarrow x_{0}$ for some $x_{0}\in \Lambda$ such that $V(x_{0})=-V_{0}$. 
From the last limit, we can find $r>0$ such that for some subsequence, still denoted by itself, it holds
$$
B_{r}(\e_{n}y_{n})\subset \Lambda, \quad \mbox{ for all } n\in \mathbb{N}.
$$
Hence, 
$$
B_{\frac{r}{\e_{n}}}(y_{n})\subset \Lambda_{\e_{n}}, \quad \mbox{ for all } n\in \mathbb{N},
$$
or equivalently,
\begin{equation}\label{ernAMPA}
\Lambda^{c}_{\e_{n}}\subset B^{c}_{\frac{r}{\e_{n}}}(y_{n}),  \quad \mbox{ for all } n\in \mathbb{N}.
\end{equation}
Now, by Lemma \ref{lem2.6AM}, there exists $R>0$ such that
$$
w_{n}(x, 0)<a, \quad \mbox{ for all } |x|\geq R \mbox{ and } n\in \mathbb{N}.
$$ 
Therefore, $v_{n}(x, 0)=w_{n}(x-y_{n}, 0)<a$ for all $x\in B^{c}_{R}(y_{n})$ and $n\in \mathbb{N}$. On the other hand, there exists $n_{0} \in \mathbb{N}$ such that  
$$
\Lambda^{c}_{\e_{n}}\subset B^{c}_{\frac{r}{\e_{n}}}(y_{n})\subset B^{c}_{R}(y_{n}), \quad \mbox{ for all } n\geq n_{0}.
$$
Consequently, 
\begin{align*}
v_{n}(x, 0)<a, \quad \mbox{ for all } x\in \Lambda^{c}_{\e_{n}} \mbox{ and } n\geq n_{0}.
\end{align*}
Then, there exists $\e_{0}>0$ such that, for all $\e \in (0, \e_{0})$, Problem \eqref{EP} admits a solution $v_{\e}$.
Invoking the weak Harnack inequality \cite[Proposition 2]{FF}, we conclude that $v_{\e}(\cdot, 0)>0$ in $\R^{N}$.

Next, we study the behavior of the maximum points of solutions to Problem \eqref{P}. 
Take $\e_{n}\rightarrow 0$ and let $(v_{n})\subset X_{\e_{n}}$ be a sequence of mountain pass solutions to \eqref{MEP}.  By Lemma \ref{lem2.5AM}, up to a subsequence, there exist $(y_{n})\subset \R^{N}$ and $w\in \x\setminus\{0\}$ such that $w_{n}(x, y):=v_{n}(x+y_{n}, y)$ strongly converges to $w$ in $\x$ as $n\ri \infty$. Moreover, there exists $x_{0}\in M$ such that $\e_{n}y_{n}\ri x_{0}$. 
Let now $q_{n}$ be a global maximum point of $w_{n}(\cdot, 0)$. Lemma \ref{lem2.6AM} and Lemma \ref{lemDELTA} guarantee that there exists $\bar{R}>0$ such that $|q_{n}|\leq \bar{R}$ for all $n\in \mathbb{N}$. Thus, $x_{n}:=q_{n}+y_{n}$ is a global maximum point of $v_{n}(\cdot, 0)$, and $\e_{n}x_{n}\ri x_{0}\in M$. This fact together with the continuity of $V$ produces 
$$
\lim_{n\rightarrow \infty} V(\e_{n} x_{n})=V(x_{0})=-V_{0}.
$$
Finally, we prove a decay estimate for $v_{n}(\cdot, 0)$. 
Using Lemma \ref{lem2.6AM}, $(f_1)$, and the definition of $g$, there exists $R_{0}>0$ sufficiently large such that
\begin{align}\label{TERESA}
g_{\e_{n}}(x+y_{n}, w_{n}(x,0))\leq \delta w_{n}(x, 0), \quad \mbox{ for all } |x|> R_{0} \mbox{ and } n\in \mathbb{N},
\end{align}
where $\delta\in (0, m^{2s}-V_{1})$ is fixed.
Arguing as in \cite{ADCDS} (see formulas $(57)$ and $(58)$ in \cite{ADCDS}), we can find a positive continuous function $\bar{w}\in H^{s}(\R^{N})$ and $R_{1}>0$ such that
\begin{align}\label{HZ2AMPA}
(-\Delta+m^{2})^{s} \bar{w}-(V_{1}+\delta)\bar{w}= 0 \quad \mbox{ in } \bar{B}^{c}_{R_{1}}
\end{align}
and
\begin{align}\label{HZ1AMPA}
0<\bar{w}(x)\leq C e^{-c |x|}, \quad \mbox{ for all } x\in \R^{N},
\end{align} 
for some $C, c>0$.
Put $R_{2}:=\max\{R_{0}, R_{1}\}$. Thus, by $(V_1)$ and \eqref{TERESA}, we have that
\begin{align}\label{HZ3AMPA}
(-\Delta+m^{2})^{s} w_{n}(\cdot, 0)-(V_{1}+\delta) w_{n}(\cdot,0) \leq 0 \quad \mbox{ in } \bar{B}^{c}_{R_{2}}. 
\end{align}
Set $b:=\min_{\bar{B}_{R_{2}}} \bar{w}>0$ and $z_{n}:=(\ell+1)\bar{w}-bw_{n}(\cdot, 0)$, where $\ell:=\sup_{n\in \mathbb{N}} |w_{n}(\cdot, 0)|_{\infty}<\infty$. 
Let us observe that $z_{n}\geq 0$ in $\bar{B}_{R_{2}}$ and that \eqref{HZ2AMPA} and \eqref{HZ3AMPA} yield
$$
(-\Delta+m^{2})^{s}z_{n}-(V_{1}+\delta)z_{n}\geq 0 \quad\mbox{ in } \bar{B}^{c}_{R_{2}}.
$$
Because $V_{1}+\delta<m^{2s}$, we can use a comparison argument (see \cite[Theorem 4.3]{ADCDS} with $\Omega=\bar{B}^{c}_{R_{2}}$) to deduce that $z_{n}\geq 0$ in $\R^{N}$.
In view of (\ref{HZ1AMPA}), there exists $C_{0}>0$ such that
\begin{align*}
0\leq w_{n}(x, 0)\leq C_{0} \, e^{-c |x|}, \quad  \mbox{ for all } x\in \R^{N} \mbox{ and } n\in \mathbb{N}.
\end{align*}
Recalling that $v_{n}(x, 0)=w_{n}(x-y_{n}, 0)$, we arrive at
\begin{align*}
v_{n}(x, 0)=w_{n}(x-y_{n},0)\leq C_{0} \, e^{-c|x-y_{n}|}, \quad \mbox{ for all } x\in \R^{N} \mbox{ and } n\in \mathbb{N}.
\end{align*}
This completes the proof of Theorem \ref{thm1}.
\end{proof}

\section{Final comments: multiple concentrating solutions to \eqref{P}}
As in \cite{ADCDS}, if we suppose that the continuous potential $V:\R^{N}\ri \R$ satisfies the following conditions:
\begin{compactenum}[$(V'_1)$]
\item  there exists $V_{0}\in (0, m^{2s})$ such that $-V_{0}:=\inf_{x\in \R^{N}}V(x)$,
\item there exists a bounded open set $\Lambda\subset \R^{N}$ such that
\begin{align*}
-V_{0}<\min_{x\in \partial \Lambda} V(x) \,\, \mbox{ and } \,\, 0\in M:=\{x\in \Lambda: V(x)=-V_{0}\},
\end{align*}
\end{compactenum}
then we obtain the next multiplicity result:
\begin{thm}\label{thm2}
Assume that $(V'_1)$-$(V'_2)$ and $(f_1)$-$(f_4)$ hold. Then, for each $\delta>0$ such that
$$
M_{\delta}:=\{x\in \R^{N}: {\rm dist}(x, M)\leq \delta\}\subset \Lambda,
$$ 
there exists $\e_{\delta}>0$ such that, for each $\e\in (0, \e_{\delta})$, Problem \eqref{P} has at least $cat_{M_{\delta}}(M)$ positive solutions. Moreover, if $u_{\e}$ denotes one of these solutions and $x_{\e}$ is a global maximum point of $u_{\e}$, then we have 
$$
\lim_{\e\rightarrow 0} V(\e x_{\e})=-V_{0}.
$$	
\end{thm}
Since the proof of Theorem \ref{thm2} is similar to the one of \cite[Theorem 1.2]{ADCDS}, we only point out the main differences.
For the proof of the critical version of  \cite[Lemma 5.1]{ADCDS}, we only need to replace in formula $(67)$ in \cite{ADCDS} the term
\begin{align*}
\int_{\La_{\e}} F(tu_{n}(x, 0))\, dx
\end{align*}
by the term
\begin{align*}
\int_{\La_{\e}} \left[F(tu_{n}(x, 0))+\frac{t^{\2}}{\2}(u_{n}^{+}(x, 0))^{\2}\right]\, dx,
\end{align*}
and use the same estimates given in the proof of \cite[Lemma 5.1-(iv)]{ADCDS}. For the analog of \cite[Corollary 1]{ADCDS}, we exploit Lemma \ref{lemma2} instead of \cite[Lemma 3.2]{ADCDS}.
For the proof of the critical version of \cite[Lemma 5.4]{ADCDS}, it suffices to replace in formula $(71)$ in \cite{ADCDS} the term
\begin{align*}
\int_{\R^{N}} F(t_{\e_{n}}\eta(|(\e_{n}x',0)|)w(x', 0))\, dx'
\end{align*}
by the term
\begin{align*}
\int_{\R^{N}} F(t_{\e_{n}}\eta(|(\e_{n}x',0)|)w(x', 0))+\frac{t^{\2}_{\e_{n}}}{\2} (\eta(|(\e_{n}x',0)|)w(x', 0))^{\2} \, dx',
\end{align*}
in formula $(72)$ in \cite{ADCDS} we substitute the term
\begin{align*}
\int_{\R^{N}} f(t_{\e_{n}}\eta(|(\e_{n}x',0)|)w(x', 0)) t_{\e_{n}}\eta(|(\e_{n}x',0)|)w(x', 0)\, dx'
\end{align*}
with
\begin{align*}
\int_{\R^{N}} [f(t_{\e_{n}}\eta(|(\e_{n}x',0)|)w(x', 0)) t_{\e_{n}}\eta(|(\e_{n}x',0)|)w(x', 0)+(t_{\e_{n}}\eta(|(\e_{n}x',0)|)w(x', 0))^{\2}]\, dx',
\end{align*}
and we replace formula $(73)$ in \cite{ADCDS} by
\begin{align*}
\|\Psi_{\e_{n}, z_{n}}\|_{\e_{n}}^{2} &\geq \int_{B_{\frac{\delta}{2}}} \frac{[f(t_{\e_{n}}w(x', 0))+(t_{\e_{n}}w(x', 0))^{\2-1} ]}{t_{\e_{n}}w(x', 0)} w^{2}(x', 0)\, dx' \nonumber \\
&\geq  t_{\e_{n}}^{\2-2} \int_{B_{\frac{\delta}{2}}} w^{\2}(x', 0)\,dx' \geq t_{\e_{n}}^{\2-2} |B_{\frac{\delta}{2}}| w^{\2}(\hat{x}, 0).
\end{align*}
Finally, for the proof of the critical version of \cite[Proposition 4]{ADCDS}, we consider Lemma \ref{Lions2} instead of  \cite[Lemma 3.4]{ADCDS}, and replace the terms 
\begin{align*}
\int_{\mathbb{R}^{N}\setminus B_{R/\e_{n}}} f(v_{n}(x, 0)) v_{n}(x, 0) \, dx,
\end{align*} 
\begin{align*}
\int_{\mathbb{R}^{N}} F(\tilde{v}_{n}(x, 0))\,dx,
\end{align*}
\begin{align*}
\int_{\mathbb{R}^{N}} F(t_{n}u_{n}(x, 0))\,dx,
\end{align*}
by
\begin{align*}
\int_{\mathbb{R}^{N}\setminus B_{R/\e_{n}}} \left[f(v_{n}(x, 0)) v_{n}(x, 0)+(v_{n}^{+}(x, 0))^{\2}\right] \, dx,
\end{align*} 
\begin{align*}
\int_{\mathbb{R}^{N}} \left[F(\tilde{v}_{n}(x, 0))+\frac{t_{n}^{\2}}{\2}(\tilde{v}_{n}^{+}(x, 0))^{\2}\right]\,dx,
\end{align*}
\begin{align*}
\int_{\mathbb{R}^{N}} \left[F(t_{n} u_{n}(x, 0))+\frac{t_{n}^{\2}}{\2}(u_{n}^{+}(x, 0))^{\2}\right] \,dx,
\end{align*}
respectively. No additional substantial modifications are necessary to deduce the required multiplicity result.

\end{document}